\documentclass[11pt]{article}
\usepackage{amsmath}
\usepackage{amsfonts,amssymb,amstext}
\usepackage{amsmath}
\usepackage{amstext}
\usepackage{amsthm}
\usepackage{graphicx}
\usepackage{setspace}
\usepackage{xcolor}
\makeatletter

\newcommand{\beq}{\begin{equation}}
\newcommand{\eeq}{\end{equation}}
\newcommand{\bea}{\begin{eqnarray}}
\newcommand{\eea}{\end{eqnarray}}
\newcommand{\beaa}{\begin{eqnarray*}}
\newcommand{\eeaa}{\end{eqnarray*}}

\newcommand{\n}{\noindent}
\newcommand{\q}{\quad}
\newcommand{\qq}{\qquad}

\newcommand{\E}{\rm I\!E}

\newcommand{\g}{\gamma}
\newcommand{\I}{\varphi}
\newcommand{\G}{\Gamma}
\newcommand{\de}{\delta}
\newcommand{\De}{\Delta}

\newcommand{\al}{\alpha}
\newcommand{\la}{\lambda}
\newcommand{\f}{\infty}
\newcommand{\vs}{\varepsilon}
\newcommand{\cd}{\cdot}
\newcommand{\si}{\sigma}

\newcommand{\be}{\beta}

\newcommand{\om}{\omega}

\newcommand{\inl}{\int_{-\pi}^{\pi}}

\newcommand {\ol} {\overline}

\newcommand {\s} {\section}
\newcommand {\sn} {\subsection}

\newtheorem{thm}{Theorem}[section]
\newtheorem{lem}{Lemma}[section]
\newtheorem{pp}{Proposition}[section]
\newtheorem{cor}{Corollary}[section]

\newtheorem{exa}{Example}[section]
\newtheorem{rem}{Remark}[section]

\newtheorem{den}{Definition}[section]

\numberwithin{equation}{section}

\newtheorem*{TA}{Theorem A}
\newtheorem*{TB}{Theorem B}

\usepackage[active]{srcltx}
\begin{document}

\title{Extensions of Rosenblatt's results on the asymptotic behavior of the
prediction error for deterministic stationary sequences\thanks{This paper is 
submitted to a special issue in memory of Professor Murray Rosenblatt}}

\author{Nikolay M. Babayan, %\thanks{ Russian-Armenian University, Yerevan, Armenia, e-mail: nmbabayan@gmail.com},
Mamikon S. Ginovyan%\thanks{Boston University, Boston, USA, e-mail: ginovyan@math.bu.edu}
\,
and Murad S. Taqqu%\thanks{Boston University, Boston, USA, e-mail: murad@bu.edu}
}

%\date{}
\date{\today}
\maketitle

\begin{abstract}
\noindent
One of the main problem in prediction theory of discrete-time second-order stationary
processes $X(t)$ is to describe the asymptotic behavior of the best linear mean
squared prediction error in predicting $X(0)$ given $ X(t),$ $-n\le t\le-1$,
as $n$ goes to infinity.
This behavior depends on the regularity (deterministic or non-deterministic)
of the process $X(t)$.
In his seminal paper {\it "Some purely deterministic processes" (J. of Math. and
Mech.,} 6(6), 801-810, 1957), M. Rosenblatt has described the asymptotic
behavior of the prediction error for discrete-time deterministic processes
in the following two cases:
(a) the spectral density $f(\la)$ of $X(t)$ is continuous and vanishes on
an interval,
(b) the spectral density $f(\la)$ has a very high order contact with zero.
He showed that in the case (a) the prediction error variance behaves
exponentially, while in the case (b), it behaves hyperbolically as $n\to\f$.
In this paper, using a new approach, we describe extensions of Rosenblatt's
results to broader classes of spectral densities.
Examples illustrate the obtained results.
\end{abstract}



\s{Introduction}
\label{Int}

\sn{The prediction problem}

Let $X(t),$ $t\in\mathbb{Z}: = \{0,\pm1,\ldots\}$, be a second-order stationary
stochastic sequence possessing a spectral density function $f(\la),$
$\la\in\Lambda: = [-\pi, \pi].$ The {\it "finite" linear prediction problem}
is as follows.

Suppose we observe a finite realization of the process $X(t)$:
$$\{X(t), \,\, -n\le t\le-1\}, \q n\in\mathbb{N}: = \{1,2, \ldots\}.$$
We want to predict the random variable $X(0)$, which is the unobserved
one-step ahead value of the process $X(t)$, using the {\it linear predictor}
$Y=\sum_{k=1}^{n}c_kX(-k)$. The coefficients
$c_k$, $k=1,2,\ldots,n$, are chosen so as to minimize
{\it the mean-squared error}: $\E\left|X(0) - Y\right|^2,$
where $\E[\cd]$ stands for the expectation operator.
If such minimizing constants $\widehat c_k:=\widehat c_{k,n}$ can be found,
then the random variable $\widehat X_n(0):=\sum_{k=1}^{n}\widehat c_kX(-k)$ is called
{\it the best linear one-step ahead predictor} of the random variable $X(0)$
based on the observed finite past: $X(-n), \ldots, X(-1)$.
The minimum mean-squared error:
$$\si_{n}^2(f): =\E\left|X(0) - \widehat X_n(0)\right|^2
=\E\left|X(0) - \sum_{k=1}^{n}\widehat c_{k,n}X(-k)\right|^2
$$
is called {\it the best linear one-step ahead prediction error} of $X(0)$ based
on the finite past of length $n$ of the process $X(t)$.

One of the main problems in prediction theory of second-order stationary
processes, called {\it the "direct" prediction problem} is to describe the
asymptotic behavior of the prediction error $\sigma_n^2(f)$ as $n\to\f$.
This behavior depends on the regularity nature (deterministic or nondeterministic)
of the observed process $X(t)$.

Observe that $\sigma_{n+1}^2(f)\leq \sigma_n^2(f)$, $n\in\mathbb{N}$,
and hence the limit of $\sigma_n^2(f)$ as $n\to\f$ exists.
Denote by $\sigma^2(f) = \sigma_\infty^2(f)$, the prediction error
of $X(0)$ by the entire infinite past: $\{X(t)$, $t\le-1\}$.

The well-known {\it Kolmogorov-Szeg\"o theorem} states that the following
limiting relation hold (see, e.g., Grenander and Szeg\"o \cite{GS}, p. 44, 183):
\bea
\label{c013}
\lim_{n\to\f}\si_n^2(f)=\si^2(f)=2\pi G(f),
\eea
where $G(f)$ is the {\it geometric mean} of $f(\la)$, namely %given by
\beq
\label{a2}
G(f): = \left \{
\begin{array}{ll}
\exp\left\{\frac1{2\pi}\inl\ln f(\la)\,d\la \right\} &
\mbox{if \, $\ln f \in {L}^1(\Lambda)$}\\
           0, & \mbox { otherwise.} \qq
           \end{array}
           \right.
\eeq
%The condition $\ln f \in {L}^1(\Lambda)$, which is equivalent to
The condition $\ln f \in {L}^1(\Lambda)$ in \eqref{a2} is equivalent to the
{\it Szeg\"o condition}:
\beq
\label{S}
\inl\ln f(\la)\,d\la > -\f
\eeq
(this equivalence follows because $\ln f(\la)\le f(\la)$ and $f(\la) \in {L}^1(\Lambda)$).
The Szeg\"o condition \eqref{S} is also called the {\it non-determinism condition}.

From the prediction point of view it is natural to distinguish the class of
processes for which we have {\it error-free prediction}, that is,
$\si^2(f)=0$. Such processes are called {\it deterministic} or {\it singular}.
Processes for which $\si^2(f)>0$ are called {\it nondeterministic}.

In view of the relations \eqref{c013} - \eqref{S} we have the following spectral
characterization of deterministic and nondeterministic processes
possessing spectral densities, known as {\it Kolmogorov-Szeg\"o alternative}:
{\it either
$$
\inl\ln f(\la)\,d\la = -\f  \, \Longleftrightarrow \, \sigma^2(f)=0 \,
\Longleftrightarrow \, X(t) \,\,  is \,\, deterministic,
$$
or else}
$$
\inl\ln f(\la)\,d\la > -\f \, \Longleftrightarrow \,
\sigma^2(f) >0 \, \Longleftrightarrow \,  X(t) \,\, is \,\, nondeterministic.
$$
Following Rosenblatt \cite{Ros}, we will say that the spectral density
$f(\la)$ has a { \it very high order of contact with zero at a point} $\la_0$ if $f(\la)$
is positive everywhere except for the point $\la_0$, due to which
the Szeg\"o condition \eqref{S} is violated.
Observe that the Szeg\"o condition \eqref{S} is connected to the
character of the singularities (zeroes) of the spectral density $f(\la)$,
and does not depend on the differential properties of $f(\la)$.
For example, for any  $\al\ge0$, the function
\beq
\label{SE}
\nonumber
f(\la)=\exp\{-|\la|^{-\al}\}
\eeq
is infinitely differentiable,
for $\al<1$ Szeg\"o condition is satisfied, and hence a stationary
process $X(t)$ with this spectral density is nondeterministic,
while for $\al\ge 1$ Szeg\"o condition \eqref{S} is violated, and $X(t)$
is deterministic.
Thus, according to the above definition, for $\al\ge 1$ the spectral density
\eqref{SE} has a very high order of contact with zero at the point $\la=0$.

Define the {\it relative prediction error} $\de_n(f)$:
\begin{equation}
\label{pp5}
\nonumber
\de_n(f): = \si^2_n(f) - \si^2(f),
\end{equation}
and observe that
$\delta_n(f)$ is nonnegative and tends to zero as $n\to\infty$.
But what about the speed of convergence of $\delta_n(f)$ to zero as $n\to\infty$?
The paper deals with this question.
Specifically, the prediction problem we are interested in is
{\it to describe the rate of decrease of $\delta_n(f)$ to zero as}
$n \to \infty,$ depending on the regularity
nature (deterministic or nondeterministic) of the observed process $X(t)$.

The prediction problem stated above goes back to classical works of
A. Kolmogorov \cite{Kol1}, \cite{Kol2}, G. Szeg\"o \cite{Sz}, \cite{Sz1}
and N. Wiener \cite{Wr1}.
It was then considered by many authors for different classes of nondeterministic
processes (see, e.g., Baxter \cite{Bax}, Devinatz \cite{Dev1}, Golinski \cite{Gol-1},
Golinski and Ibragimov \cite{GI} Grenander and Rosenblatt \cite{GR1}, Grenander
and Szeg\"o \cite{GS}, Helson and Szeg\"o \cite{HS}, Hirshman \cite{Hir},
Ibragimov \cite{I-2}, Ibragimov and Solev \cite{IS}, Inoue \cite{In2},
Pourahmadi \cite{Po}, Rozanov \cite{R}, and reference therein).
More references can be found in the survey papers Bingham \cite{Bin1} and
Ginovyan \cite{G-4}.

We focus in this paper on deterministic processes, that is, when  $\si^2(f) =0$.
This case is not only of theoretical interest, but is also important from the point
of view of applications.
For example, as pointed out by M. Rosenblatt \cite{Ros},
situations of this type arise in Neumann's theoretical model of
storm-generated ocean waves. Such models are also of interest
in meteorology, because the meteorological spectra often have a gap in
the mesoscale region (see, e.g., Fortus \cite{For}).

Only few works are devoted to the study of the speed of convergence of
$\de_n(f)=\si^2_n(f)$ to zero as $n\to\infty$, that is, the asymptotic behavior of
the prediction error for deterministic processes.
One needs to go back to the classical work of M. Rosenblatt \cite{Ros}.
Using the technique of orthogonal polynomials on the unit circle and Szeg\"o's
results, M. Rosenblatt investigated the asymptotic behavior of the prediction
error $\de_n(f)=\si^2_n(f)$ for discrete-time deterministic processes in the
following two cases:
\begin{itemize}
\item[(a)]
the spectral density $f(\la)$ is continuous and vanishes on an entire interval,
\item[(b)]
the spectral density $f(\la)$ is positive away from $\la=0$ and has a
very high order of contact with zero at $\la=0$, so that the
Szeg\"o condition \eqref{S} is violated.
\end{itemize}

Later the problem (a) was studied by  Babayan  \cite{Bb-1}, \cite{Bb-2},
(see also Davisson \cite{Dav1} and Fortus \cite{For}), where some generalizations and
extensions of Rosenblatt's results have been obtained.
\paragraph{Some notation.}
Throughout the paper we will use the following standard notation.\\
The standard symbols $\mathbb{N}$, $\mathbb{Z}$, $\mathbb{R}$ and $\mathbb{C}$
denote the sets of natural, integer, real and complex numbers, respectively.
Also, we denote
$\mathbb{Z}_+: = \{0,1,2,\ldots\}$, $\Lambda: = [-\pi, \pi],$
$\mathbb{T}:=\{z\in \mathbb{C}: \, |z|=1\}$.
The letters $C$, $c$, $M$ and $m$ with or without indices are used
to denote positive constants, the values of which can vary from line to line.
For a set $E$ by $\ol E$ we denote the closure of $E$.
We devote by $L_p:=L_p(\Lambda)$ ($p\geq $1) the Lebesgue space,
by $||\cdot||_p$ the norm in $L_p$, and by $\mu$ the Lebesgue measure on $\mathbb{R}$.
For two functions $f(\la)\geq0$ and $g(\la)\geq0$, $\la \in \Lambda$,
we will write $f(\lambda){\sim}g(\lambda)$ as ${\lambda\to\lambda_0}$
if $\lim_{\lambda\to\lambda_0}\frac{f(\lambda}{g(\lambda)}=1$,
and $f(\lambda){\simeq}g(\lambda)$ as ${\lambda\to\lambda_0}$ if
$\lim_{\lambda\to\lambda_0}\frac{f(\lambda}{g(\lambda)}=c>0$.
We will use similar notation for sequences: for two sequences
$\{a_n\geq0, n\in\mathbb{N}\}$ and $\{b_n>0, n\in\mathbb{N}\}$,
we will write $a_n\sim b_n$ if $\lim_{n\to\f}\frac{a_n}{b_n}=1$,
$a_n{\simeq} b_n$ if $\lim_{n\to\f}\frac{a_n}{b_n}=c>0$,
$a_n=O(b_n)$ if $\frac{a_n}{b_n}$ is bounded,
and $a_n=o(b_n)$ if $\frac{a_n}{b_n}\to0$ as $n\to\f$.\\
\\
We start by describing Rosenblatt's results concerning the asymptotic behavior
of the prediction error $\si^2_n(f)$, obtained in Rosenblatt \cite{Ros} for the
above stated cases (a) and (b).

\sn {Rosenblatt's results about speed of convergence}

For the case (a) above, that is, when the spectral density $f(\la)$ is continuous
and vanishes on an entire interval, M. Rosenblatt proved in \cite{Ros}
that the prediction error $\si^2_n(f)$ decreases to zero exponentially
as $n\to\f$. More precisely, M. Rosenblatt proved in \cite{Ros} the following theorem
concerning speed of convergence of $\de_n(f)=\si^2_n(f)$ to zero as $n\to\f$.

\begin{TA}[Rosenblatt's first theorem]
\label{R1}
Let the spectral density $f(\la)$ of a discrete-time stationary
process $X(t)$ be positive and continuous on the interval
$$(\pi/2-\alpha, \pi/2+\alpha), \q 0<\alpha<\pi,$$
and zero elsewhere. Then the
prediction error $\si^2_n(f)$ approaches zero
exponentially as $n\to\f$.
More precisely, the following asymptotic relation holds:
\beq
\label{nd1}
\de_n(f): = \si^2_n(f)\simeq\left(\sin\frac\alpha2\right)^{2n+1}
\q {\rm as} \q n\to\f.
\eeq
\end{TA}%\end{thm}
Thus, when the spectral density $f(\la)$ is continuous and vanishes on an
entire interval, then the prediction error approaches zero as fast as the
$(2n + 1)^{\text{th}}$ power of a positive number less than one.
Notice that \eqref{nd1} implies that
\beq
\label{nd2}
\lim_{n\to\f}\sqrt[n]{\sigma_n(f)}  =\sin\frac\alpha2.
\eeq
Since $0<\sin\frac\alpha2<1$ we have $\si_n(f)\to 0$ as $n\to\f$, and
the relation \eqref{nd2} is a convenient way to characterize the speed
of convergence of $\si_n(f)$ to zero as $n\to\infty$.
Viewed from the perspective of the unit circle in the complex plane,
the arc $(\pi/2-\alpha, \pi/2+\alpha)$ under consideration is of size $2\alpha$.
This arc plays an important role in the sequel.

Theorem \ref{BB1} below extends Theorem A.

Concerning the case (b) above,  M. Rosenblatt proved in \cite{Ros} that if the
spectral density $f(\la)$ of a stationary process $X(t)$ is positive away
from zero, and has a very high order contact with zero at $\la=0$, so that
the Szeg\"o condition \eqref{S} is violated, then the prediction error
$\de_n(f) = \si^2_n(f)-\si^2(f)=\si^2_n(f)$ decreases to zero {\it hyperbolically} as $n\to\f$,
that is,
$$\de_n(f)=\si^2(f) \simeq n^{-a} \,\, (a>0) \q {\rm as} \q n\to\f.$$
More precisely, the deterministic process $X(t)$ considered in \cite{Ros} has
the spectral density
\beq
\label{nd4}
f_a(\la):=\frac{e^{(2\la-\pi)\varphi(\la)}}{\cos\la\left(\pi\varphi(\la)\right)},
%\ |\sin(\la)|,
\q f_a(-\la)=f_a(\la),\q 0\leq\la\leq\pi,
\eeq
where $\varphi(\la)=\frac a2\cot\la$ and $a$ ($a>0$) is a fixed parameter.
In Rosenblatt \cite{Ros} it is observed that
\beq
\label{nd5}
f_a(\la)\sim 2\exp\left\{-\frac{a\pi}{|\la|}\right\}|\sin(\la)|\q {\rm as}
\q \la\to0,
\eeq
so that $f_a(\la)$ has a very high order contact with zero only at $\la=0$,
and the Szeg\"o condition \eqref{S} is violated.

In \cite{Ros}, using the technique of orthogonal polynomials on the unit circle
and Szeg\"o's results, M. Rosenblatt proved the following theorem.
\begin{TB}[Rosenblatt's second theorem]
\label{R2}
Suppose that the process $X(t)$ has spectral density $f_a(\la)$ given by \eqref{nd4}.
Then the following asymptotic formula for the prediction error $\de_n(f) = \si^2_n(f)$
holds:
\beq
\label{nd6}
\de_n(f_a) = \si^2_n(f_a)\sim\frac{\Gamma^2\left(\frac{a+1}2\right)}
{\pi 2^{2-a}} \ n^{-a}%\sim n^{-a}
\q {\rm as} \q n\to\f.
\eeq
\end{TB}%\end{thm}
In this paper, using an approach, different from the one applied in Rosenblatt
\cite{Ros}, we extend Theorems A and B to broader classes of
spectral densities.

Concerning Theorem A, we describe an extension of the asymptotic relation \eqref{nd2}
to the case of several arcs, without having to stipulate continuity of the
spectral density $f(\la)$.

As for the extension of Theorem B, we first prove that if the spectral
density $f(\la)$ is such that the sequence $\{\si_n(f)\}$ is weakly varying
(a term defined in Section \ref{pre}) and if, in addition, $g(\la)$ is the
spectral density of a nondeterministic process satisfying some conditions,
then the sequences $\{\si_n(fg)\}$ and $\{\si_n(f)\}$  have the same asymptotic
behavior as $n\to\f$, up to some positive multiplicative factor $G(g)$.
This allows us to derive the asymptotic behavior of $\{\si_n(fg)\}$ from that of $\{\si_n(f)\}$.

Using this result, we obtain the following extension of Theorem B:
if the spectral density $f(\la)$ has the form $f(\la)=f_a(\la)g(\la)$,
where $f_a(\la)$ is as in (\ref{nd4}) and $g(\la)$ is the spectral density
of a nondeterministic process, then $\de_n(f)=\si^2(f)\simeq n^{-a}$ as $n\to\f$.

The remainder of the paper is organized as follows.
In Section \ref{fpp} we present formulas for the finite prediction error $\si_n(f)$,
and state some preliminary results.
Section \ref{d1} is devoted to the extension of the Rosenblatt's first theorem
(Theorem A). In Section \ref{d2} we extend Rosenblatt's second theorem
(Theorem B).

\s{Formulas for the prediction error}
\label{fpp}

In this section we present formulas for the finite prediction error
$\si_{n}^2(f)$ and state some preliminary results, which will be used in the sequel.

Suppose we have observed the values $X(-n), \ldots, X(-1)$ of a centered,
real-valued stationary process $X(t)$ with spectral density $f(\la)$.
The {\it one-step ahead linear prediction problem}
in predicting a random variable $X(0)$ based on the observed values
$X(-n), \ldots, X(-1)$ involves finding constants
$\widehat c_k:=\widehat c_{k,n}$, $k=1,2,\ldots,n$,
that minimize the one-step ahead prediction error variance:
\bea
\label{OS1}
\si_{n}^2(f):
=\min_{\{c_k\}}\E\left|X(0) - \sum_{k=1}^{n}c_kX(-k)\right|^2
=\E\left|X(0) - \sum_{k=1}^{n}\widehat c_kX(-k)\right|^2,
\eea
where $\E[\cd]$ stands for the expectation operator.
Using Kolmogorov's isometric isomorphism $V:\, X(t)\leftrightarrow e^{it\la}$,
in view of (\ref{OS1}), for the prediction error $\sigma_n^2(f)$ we can write
\bea
\label{OS2}
\nonumber
\sigma_n^2(f)& =&
\min_{\{c_k\}}\inl\left|1 -\sum_{k=1}^{n}c_ke^{-ik\la}\right|^2 f(\la)d\la
=\min_{\{c_k\}}\inl\left|e^{in\la} -\sum_{k=1}^{n}c_ke^{i(n-k)\la}\right|^2f(\la)d\la\\
&=&\min_{\{q_n\in \mathcal{Q}_n\}} \inl\left|q_n(e^{i\la})\right|^2f(\la)d\la,
\eea
where
\beq
\label{Q_n}
\mathcal{Q}_n: =\left\{q_n:  q_n(z)=\sum_{k=0}^nc_kz^{n-k}, \, c_0=1\right\}
\eeq
stands for the set of monic polynomials of degree $n$, that is,
with coefficient of the leading term equal to 1.

Thus, the problem of finding $\sigma_n^2(f)$ becomes to the problem of finding
the solution of the minimum problem \eqref{OS2}, \eqref{Q_n}.

The polynomial $p_n(z):=p_n(z,f)$ which solves the minimum problem \eqref{OS2}, \eqref{Q_n}
is called the {\it optimal polynomial} for $f(\la)$ in the class $\mathcal{Q}_n$.
This minimum problem was solved by G. Szeg\"o (see, e.g., Grenander and Szeg\"o \cite{GS},
Section 2.2, p. 38) by showing that the optimal polynomial $p_n(z,f)$ exists,
is unique and can be expressed
in terms of {\it orthogonal polynomials}
$\I_n(z)= \I_n(z;f)$, $n\in \mathbb{Z}_+$, on the unit circle $\mathbb{T}$
with respect to the spectral density $f(\la)$.

The system of orthogonal polynomials:
\beq
\label{orp}
\{\I_n(z) = \I_n(z;f), \q z=e^{i\la}, \q n\in \mathbb{Z}_+\}
\eeq
is uniquely determined by the following two conditions:
\begin{itemize}
\item[(i)] \,$\I_n(z) = \kappa_n(f)z^n + {\rm lower \,\,  order \,\, terms}$

is a polynomial of degree $n$, in which the coefficient
$\kappa_n = \kappa_n(f)$ of $z^n$ is real and positive;

\item[(ii)] for arbitrary nonnegative integers $k$ and $j$
\beq
\label{op4}
\nonumber
\frac1{2\pi}\inl \I_k(z)\ol {\I_j(z)}f(\la)d\la =
\de_{kj} = \left \{
           \begin{array}{ll}
           1, & \mbox { for  $ k = j$}\\
           0, & \mbox { for  $k \ne j,$} \q
           \end{array}
           \right.   z = e^{i\la}.
\eeq
\end{itemize}
The next result by Szeg\"o, which solves the minimum problem \eqref{OS2}, \eqref{Q_n},
can be found in Szeg\"o \cite{Sz1}, p. 298 (see, also, Grenander and Szeg\"o \cite{GS}, p. 38).
\begin{pp}
\label{Sz2}
The optimal polynomial for $f(\la)$ in the class $\mathcal{Q}_n$, that is,
the polynomial $p_n(z):=p_n(z,f)$ which solves the minimum problem \eqref{OS2}, \eqref{Q_n}
is given by
\beq
\label{opp}
\nonumber
p_n(z)=\kappa^{-1}_n(f)\I_n(z),
\eeq
and the minimum in \eqref{OS2} is equal to $\kappa^{-2}_n(f)$, where $\I_n(z)$
is as in \eqref {orp} and $\kappa_n(f)$ is the leading coefficient of $\I_n(z)$.
\end{pp}
Thus, for the prediction error $\sigma_n^2(f)$ we have the following formula:
\bea
\label{mm1}
\si^2_n(f)=
\min_{\{q_n\in \mathcal{Q}_n\}}
\inl\left|q_n(e^{i\la})\right|^2f(\la)d\la=
\inl\left|p_n(e^{i\la},f)\right|^2f(\la)d\la=\kappa^{-2}_n(f).
\eea
\begin{rem}
\label{Rem22}
{\rm Define
\beq
\label{Q_n*}
\mathcal{Q}^*_n: =\left\{q_n:  q_n(z)=\sum_{k=0}^nc_kz^{n-k}, \, c_n=1\right\},
\eeq
and observe that the classes of polynomials $\mathcal{Q}_n$ and $\mathcal{Q}^*_n$
defined in \eqref{Q_n} and \eqref{Q_n*}, respectively, differ by normalization:
in \eqref{Q_n*} we have $ c_n=1$, while in \eqref{Q_n} we have $ c_0=1$.
Also, the optimal polynomials $p_n(z,f)$ and $p^*_n(z,f)$ for $f(\la)$ in the
classes $\mathcal{Q}_n$ and $\mathcal{Q}^*_n$ are related by the equality:
$p^*_n(z,f)=z^{n}\ol p_n(1/z,f)$, that is, $p^*_n(z,f)$ is the reciprocal polynomial
for $p_n(z,f)$, and, we have $||p^*_n||_2=||p_n||_2$.
Thus, for the prediction error $\si^2_n(f)$ we have the following formula in terms
of the optimal polynomial $p^*_n(z,f)$:
\bea
\label{mmm1}
\si^2_n(f)=
\min_{\{q_n\in \mathcal{Q}^*_n\}}
\inl\left|q_n(e^{i\la})\right|^2f(\la)d\la=
\inl\left|p^*_n(e^{i\la},f)\right|^2f(\la)d\la.
\eea}
\end{rem}
\begin{rem}
\label{rW}
{\rm Consider the mapping $W:(-\pi,\pi]\longleftrightarrow\mathbb{T}$,
given by formula
\beq
\label{W1}
W(\la)=e^{i\la}, \q \la\in (-\pi,\pi].
\eeq
The mapping $W$ establishes a one-to-one correspondence between the interval $(-\pi,\pi]$
and the unit circle $\mathbb{T}$, and allows us to pass from the interval $(-\pi,\pi]$ to the
unit circle $\mathbb{T}$ and vice versa. Also, observe that by means of this mapping the
spectral density $f$ and the Lebesgue measure $\mu$ on  $(-\pi,\pi]$ generate on the unit circle
$\mathbb{T}$ a function $f_\mathbb{T}$ and a measure $\mu_\mathbb{T}$, respectively as follows:
\beq
\label{W2}
f_\mathbb{T}(z)=f(\la), \, \,\, z=e^{i\la}, \,\, \la\in (-\pi,\pi]\q {\rm and} \q
\mu_\mathbb{T}(E)=\mu(W^{-1}(E)), \q E\subset\mathbb{T},
\eeq
where $W^{-1}(E)$ is the preimage of $E\subset\mathbb{T}$ under the mapping $W$:
\beq
\label{W-1}
W^{-1}(E):=\{\lambda\in (-\pi,\pi]: \, e^{i\lambda} \in E\}.
\eeq
The measure $\mu_\mathbb{T}(E)$ is called a {\it linear measure} of the set $E\subset\mathbb{T}$.
With this notation, in view of \eqref{mm1} and \eqref{W2}, for the prediction error $\si^2_n(f)$
we have the following formula:
\bea
\label{Wmm1}
\si^2_n(f)=
\int_\mathbb{T}\left|p_n(z,f)\right|^2f_\mathbb{T}(z)d\mu_\mathbb{T}, \q z=e^{i\la}.
\eea
}
\end{rem}
\begin{pp}
\label{pp3}
The prediction error $\sigma_n^2(f)$ possesses the following properties.
\begin{itemize}
\item[(a)]
The sequence $\{\sigma_n^2(f), \, n\in\mathbb{N}\}$ is non-increasing
in $n$: $\sigma_{n+1}^2(f)\leq\sigma_{n}^2(f)$.
\item[(b)]
$\sigma_n^2(f)$ is a non-decreasing functional of $f(\lambda)$:
\bea
\label{OSm2K}
\sigma_{n}^2(f_1)\leq\sigma_n^2(f_2) \q {\rm when}\q f_1(\lambda)\leq f_2(\lambda),
\q \lambda\in \Lambda=[-\pi,\pi].
\eea
\item[(c)]
If $f(\la)=g(\la)$ almost everywhere on $[-\pi,\pi]$, then
\beq
\label{2c0}
\si_n(f) = \si_n(g).
\eeq
Thus, the change of values of the spectral density $f(\la)$ on an arbitrary subset
of $[-\pi,\pi]$ of measure zero (in particular, on a subset consisting of a finite
or countable number of points) does not affect the value of $\si_n(f)$.
\end{itemize}
\end{pp}
\begin{proof}
The assertion (a) immediately follows from \eqref{mmm1} and the obvious embedding $\mathcal{Q}_n^*\subset\mathcal{Q}_{n+1}^*$.

To prove assertion (b), observe that by the definition of optimal polynomials
$p_n(z,f_1)$ and $p_n(z,f_2)$, corresponding to spectral densities $f_1$ and $f_2$,
respectively, we have
\beaa
\sigma_n^2(f_1)&=&\inl\left|p_n(e^{i\la},f_1)\right|^2 f_1(\la)d\la
\leq \inl\left|p_n(e^{i\la},f_2)\right|^2 f_1(\la)d\la\\
&\leq& \inl\left|p_n(e^{i\la},f_2)\right|^2 f_2(\la)d\la =\sigma_n^2(f_2),
\eeaa
and \eqref{OSm2K} follows.
In the last relation the first inequality follows from the optimality
of the polynomial $p_n(z,f_1)$, while the second inequality follows from
assumption that $f_1(\lambda)\leq f_2(\lambda)$, $\lambda\in \Lambda$.

To prove assertion (c), we split the segment $\Lambda=[-\pi,\pi]$ into
three subsets: $\Lambda=E\cup E_+\cup E_-$, where
$E:=\{\lambda\in \Lambda: \, f(\la)= g(\la)\}$,
$E_+:=\{\lambda\in \Lambda: \, f(\la)> g(\la)\}$, and
$E_-:=\{\lambda\in \Lambda: \, f(\la)< g(\la)\}$.
Taking into account that $f(\la)= g(\la)$ almost everywhere, we have
\beq
\label{2c1}
\mu(E_+)=\mu(E_-)=0.
\eeq
Define the function
\beq
\label{2c2}
h(\la): = \max\{f(\la), g(\la)\}=\left \{
\begin{array}{ll}
f(\la)\q\mbox{ if \, $\la \in E\cup E_+$}\\
g(\la)\q\mbox{ if \, $\la \in E_-$},\qq
           \end{array}
           \right.
\eeq
and observe that for all $\lambda\in \Lambda$
\beq
\label{2c3}
f(\la)\leq h(\la) \q {\rm and} \q g(\la)\leq h(\la).
\eeq
Besides, in view of \eqref{2c1}, for almost all $\lambda\in \Lambda$ we have
\beq
\label{2c4}
h(\la) = f(\la)= g(\la).
\eeq
We now show that
\beq
\label{2c5}
\si_n(h) = \si_n(f).
\eeq
Indeed, from the first inequality in \eqref{2c3} and \eqref{OSm2K}
it follows that $\si_n^2(h) \geq \si_n^2(f)$. On the other hand, we have
\beaa
\sigma_n^2(h)&=&\inl\left|p_n(e^{i\la},h)\right|^2 h(\la)d\la
\leq \inl\left|p_n(e^{i\la},f)\right|^2 h(\la)d\la\\
&=& \inl\left|p_n(e^{i\la},f)\right|^2 f(\la)d\la =\sigma_n^2(f).
\eeaa
In the last relation the first inequality follows from the optimality
of the polynomial $p_n(z,h)$, while the second equality follows from
\eqref{2c4}. Thus, \eqref{2c5} is proved.
Similar arguments can be applied to show that
\beq
\label{2c6}
\si_n(h) = \si_n(g).
\eeq
A combination of \eqref{2c5} and \eqref{2c6} yields \eqref{2c0}.
\end{proof}

\section{An extension of Rosenblatt's first theorem}
\label{d1}

In this section, using some results from geometric function theory,
we extend Rosenblatt's first theorem (Theorem A) to a broader class
of deterministic processes, possessing spectral densities that vanish on a set of
positive Lebesgue measure. More precisely, we extend the asymptotic relation
\eqref{nd2} to the case of several arcs, without having to stipulate continuity
of the spectral density $f(\la)$.
Besides, we obtain necessary as well as sufficient conditions for the exponential
decay of the prediction error $\si_{n}(f)$ as $n\to\f$.
Also, we calculate the transfinite diameter of some subsets of the unit circle,
and thus, obtain explicit asymptotic relations for the prediction error $\si_n(f)$
similar to the Rosenblatt's relation \eqref{nd2}.

To state the corresponding results we first introduce some metric characteristics
of compact (bounded closed) sets in the plane, such as, the transfinite diameter,
the Chebyshev constant and the capacity, and discuss some properties of these
characteristics.

\sn{Some metric characteristics of bounded closed sets in the plane}
\label{DCC}

One of the fundamental result of geometric complex analysis is the
classical theorem by Fekete and Szeg\"o, stating that for any compact set $F$
in the complex plane  $\mathbb{C}$ the transfinite diameter, the Chebyshev constant
and the capacity of $F$ {\em coincide}, although they are defined from very
different points of view. Namely, the transfinite diameter of the set $F$ characterizes
the asymptotic size of $F$, the Chebyshev constant of $F$ characterizes the
minimal uniform deviation of a monic polynomial on $F$, and the capacity of $F$
describes the asymptotic behavior of the Green function at infinity.
For the definitions and results stated in this subsection we refer the
reader to the following references: Fekete \cite{F}, Goluzin \cite{GMG}, Chapter 7,
Saff \cite{Saf}, Szeg\"o \cite{Sz1}, Chapter 16, and Tsuji \cite{Tsu}, Chapter III.

\vskip2mm
\n \underline{{\it Transfinite diameter.}}
Let $F$ be a compact (bounded closed) set in the complex plane $\mathbb{C}$.
Given a natural number $n\geq 2$ and points $z_1, \ldots, z_n \in F$, we define
\beq
\label{td1}
d_n(F):=\max_{z_1,\ldots,z_n\in F}\left[\prod_{1\le j<k
\le n}^n|z_j-z_k|\right]^{2/[n(n-1)]},
\eeq
which is the maximum of products of distances between the
$ \begin{pmatrix} n\\2\end{pmatrix}=n(n-1)/2$
pairs of points $z_k$, $k=1,\ldots,n$, as the points $z_k$ range over the set $F$.
Any system of points $\mathbb{F}_n:=\{z_{1n},\ldots z_{nn}\}$
for which the maximum in \eqref{td1} is attained is called an {\it $n$-point
Fekete set} for $F$, and the points $z_{kn}$ in $\mathbb{F}_n$ are called
{\it Fekete points} of $F$.\\
Note that $d_2(F)$ is the diameter of $F$, while $d_3(F)$ measures
its "spread" of $F$. The quantity $d_n(F)$ is called the $n$th
{\it transfinite diameter} of the set $F$.
It can be shown (see, e.g., Goluzin \cite{GMG}, Section 7.1, p. 294)
that $d_n(F)$ decreases and does not exceed the
diameter $d_2(F)$ of $F$, implying that $d_n(F)$ has a finite limit as $n\to\f$.
This limit, denoted by $d_\f(F)$, is called the {\it transfinite diameter} of
$F$. Thus,
\beq
\label{td2}
d_\f(F):=\lim_{n\to\f}d_n(F),
\eeq
where $d_n(F)$ is as in (\ref{td1}).
\vskip2mm
\n \underline{{\it Chebyshev constant.}}
For a bounded closed set $F$ in the complex plane $\mathbb{C}$, we put
\beq
\label{td4}
\nonumber
m_n(F): = \inf\max_{z\in F}|q_n(z)|,
\eeq
where the infimum is taken over all monic polynomials $q_n(z)$ from the
class $\mathcal{Q}_n$, where $\mathcal{Q}_n$ is as in \eqref{Q_n}.
Then there exists a unique monic polynomial $T_n(z,F)$ form the class $\mathcal{Q}_n$,
called the {\it Chebyshev polynomial} of $F$ of order $n$, such that
\beq
\label{td6}
m_n(F)=\max_{z\in F}|T_n(z,F)|.
\eeq
Fekete \cite{F} proved that $\lim_{n\to\f}(m_n(F))^{1/n}$ exists.
This limit, denoted by $\tau(F)$, is called the {\it Chebyshev constant}
for the compact set $F$. Thus,
\beq
\label{td7}
\tau(F): = \lim_{n\to\f}(m_n(F))^{1/n}.
\eeq

\vskip2mm
\n \underline{{\it Capacity}}.
Let $F$ be a closed bounded set in the complex plane  $\mathbb{C}$,
and let $D_F$ denote the complementary domain to $F$, containing $\f$
as an interior point. If the boundary $\G:=\partial D_F$ of the domain $D_F$
consists of a finite number of rectifiable Jordan curves, then for the domain $D_F$
can be constructed a Green function $G_{D_F}(z,\f)$ with a pole at infinity.
This function is harmonic everywhere in $D_F$, except at the point $z=\f$,
is continuous including the boundary $\G$ and vanishes on $\G$.
It is known that in a vicinity of the point $z=\f$ the function $G_{D_F}(z,\f)$ admits
the representation (see, e.g.,  Goluzin \cite{GMG}),  p. 309-310):
\beq
\label{tdG10}
G_F(z,\f)= \ln|z|+\g+ O(z^{-1}) \q {\rm as} \q z\to\f.
\eeq
The number $\g$ in \eqref{tdG10} is called the {\it Robin's constant} of the domain $D_F$,
and the number
\beq
\label{tdG11}
C(F):= e^{-\g}
\eeq
is called the  {\it capacity} (or the {\it logarithmic capacity}) of the set $F$.

\vskip2mm
Now we are in position to state the above mentioned fundamental result of
geometric complex analysis, due to M. Fekete and G. Szeg\"o
(see, e.g., Goluzin \cite{GMG}, Section 7.1, p. 197, Saff \cite{Saf}, and
Tsuji \cite{Tsu}, p. 73).

\begin{pp}[Fekete - Szeg\"o theorem]
\label{FS1}
For any compact set $F\subset \mathbb{C}$, the transfinite diameter $d_\f(F)$
defined by \eqref{td2}, the Chebyshev constant $\tau(F)$ defined by \eqref{td7},
and the capacity $C(F)$ defined by \eqref{tdG11} coincide, that is,
\beq
\label{td11}
d_\f(F)=C(F)=\tau(F).
\eeq
\end{pp}
\begin{rem}
\label{br0}
{\rm It what follows, we will use the term "transfinite diameter" and the notation $\tau(F)$
for \eqref{td11}.}
\end{rem}
In only very few cases can the transfinite diameter
(and hence, the capacity and the Chebyshev constant) be exactly calculated.

In the next proposition we list a number of properties of the transfinite
diameter (and hence, of the capacity and the Chebyshev constant),
which will be used later.
\begin{pp}
\label{pp1}
The transfinite diameter (and hence, the capacity and the Chebyshev constant)
possesses the following properties.
\begin{itemize}
\item[(a)]
The transfinite diameter is monotone, that is, for any closed sets $F_1$ and $F_2$ with
$F_1\subset F_2$, we have $\tau(F_1)\leq \tau(F_2)$ (see, e.g., Saff \cite{Saf}, p. 169,
Tsuji \cite{Tsu}, p. 56).
\item[(b)]
If a set $F_1$ is obtained from a compact set $F\subset \mathbb{C}$ by
a linear transformation, that is, $F_1:=aF+b=\{az+b :\, z\in F\}$, then
$\tau(F_1)=|a|\tau(F)$. In particular, the transfinite diameter $\tau(F)$
is invariant with respect to parallel translation and rotation of $F$
(see, e.g., Goluzin \cite{GMG}, p. 298, Saff \cite{Saf}, p. 169,
Tsuji \cite{Tsu}, p. 56).
\item[(c)]
(Fekete theorem). Let $F$ be a bounded closed set in the complex $w$-plane $\mathbb{C}$,
and let $p(z):=p_n(z)=z^n+c_1z^{n-1}\cdots+c_n$ be an arbitrary monic
polynomial of degree $n$. Let $F^*$ be the preimage of $F$
in the $z$-plane under the mapping $w=p(z)$, that is, $F^*$ is the set of all
points $z\in\mathbb{C}$ such that $w:=p(z)\in F$. Then (see, e.g., Goluzin \cite{GMG}, p. 299,
Saff  \cite{Saf}, p. 186):
\beq
\label{F0}
\tau(F^*)=[\tau(F)]^{1/n}.
\eeq
\item[(d)]
The transfinite diameter of an arbitrary circle of radius $R$
is equal to its radius $R$. In particular, the transfinite diameter of the unit
circle $\mathbb{T}$ is equal to 1 (see, e.g., Tsuji \cite{Tsu}, p. 84).
\item[(e)]
The transfinite diameter of an arc $\G_\al$ of a circle of radius $R$
with central angle $\al$ is equal to $R\sin\frac\al 4$.
In particular, for the unit circle $\mathbb{T}$, we have
$\tau(\G_\al)=\sin\frac\al 4$ (see Tsuji \cite{Tsu}, p. 84).
\item[(f)]
The transfinite diameter of an arbitrary line segment $F$ is equal to one-fourth its length,
that is, if $F:=[a, b]$, then (see, e.g., Saff \cite{Saf}, p. 169, Tsuji \cite{Tsu}, p. 84):
\beq
\label{ts}
\tau(F) = \tau([a, b])=\frac{b-a}4.
\eeq

\end{itemize}
\end{pp}

\sn{An extension of Rosenblatt's first theorem}
\label{ss3.2}
We are now in position to state the main results of this section.
In what follows, we use the following notation.
By $S_f^0$ we denote the set of zeros of the spectral density $f(\la)$, that is,
\beq
\label{S0}
S_f^0:=\{\la\in\Lambda: \,\, f(\la)= 0\}.
\eeq
By $S_f$ we denote the support of the spectral density $f(\la)$, that is,
\beq
\label{Spp}
S_f:=\{\la\in\Lambda: \,\, f(\la)> 0\}.
\eeq
By $E_f$ we denote the spectrum of the process $X(t)$, which is the image of
the support $S_f$ under the mapping $W$ (see \eqref{W1}), that is,
\beq
\label{Sp}
E_f:=W(S_f)=\{e^{i\la}: f(\la)> 0\},
\eeq
and by $\ol E_f$ we denote the closure of the set $E_f$.

Our first theorem extends Rosenblatt's first theorem (Theorem A).
More precisely, the result that follows extends the asymptotic
relation \eqref{nd2} to the case of several intervals (arcs), without having
to stipulate continuity of the spectral density $f(\la)$.

\begin{thm}
\label{BB1}
Let the support $S_f$ of the spectral density $f(\la)$ of the process $X(t)$
consist of a finite number of intervals of the segment $[-\pi,\pi]$.
Then the sequence $\{\sqrt[n]{\sigma_n(f)} \}$ converges, and
\beq
\label{bb2}
\lim_{n\to\f}\sqrt[n]{\sigma_n(f)}  =\tau(\ol E_f),
\eeq
where $\tau(\ol E_f)$ is the transfinite diameter of the closure of the spectrum $E_f$,
consisting of the corresponding finite number of closed arcs of the unit circle $\mathbb{T}$.
\end{thm}
\begin{rem}
{\rm
A version of Theorem \ref{BB1} was first proved in Babayan \cite{Bb-1}
(see also Babayan \cite{Bb-2}).
Here we will give a simplified proof of this result.}
\end{rem}
\begin{rem}
\label{rm3.4}
{\rm
It can be shown that under some natural additional conditions,
the asymptotic relation \eqref{bb2} remains valid in the case where the support
$S_f$ of the spectral density $f(\la)$ consists of a countable number of intervals
of the segment $[-\pi,\pi]$ (see  Babayan \cite{Bb-1}).}
\end{rem}
\begin{rem}
\label{rm3.5}
{\rm
In Theorem A we have
$$\ol E_f:=\{e^{i\la}: \la\in [\pi/2-\alpha, \pi/2+\alpha]\},$$
which represents a closed arc of length $2\al$, and, according to Proposition \ref{pp1}(e),
we have $\tau(\ol E_f)=\sin(2\alpha/4)=\sin(\alpha/2)$.
Thus, the asymptotic relation \eqref{nd2} is a special case of \eqref{bb2}.}
\end{rem}

In what follows, we will need the following definition, which characterizes
the rate of variation of a sequence compared with a geometric progression.
\begin{den}
\label{ekd1}
(a) A sequence $\{a_n, \, n \in\mathbb{N}\}$ of nonnegative numbers is
said to be exponentially neutral if
\beq
\label{en}
\lim_{n \rightarrow \infty} \sqrt[n]{a_n} =1.
\eeq
(b) A sequence $\{b_n, \, n \in\mathbb{N}\}$ of nonnegative numbers
is said to be exponentially decreasing if
\begin{equation}
\label{em2}
\limsup_{n \rightarrow \infty} \sqrt[n]{b_n} <1.
\end{equation}
\end{den}
For instance, the sequence $\{a_n=n^\al, \, \al\in\mathbb{R}, \,
n \in\mathbb{N}\}$ is exponentially neutral because
$\log \sqrt[n]{n^\alpha} = \frac{\alpha}{n} \log {n}\rightarrow 0$ as $n\to\f$.
The geometric progression $\{b_n=q^n, \, 0<q<1, \, n \in\mathbb{N}\}$
is exponentially decreasing because
$\sqrt[n]{b_n} = q^{n/n} = q<1$.
The sequence $\{b_n=n^\al q^n, \, \al\in\mathbb{R}, \, 0<q<1, \,
n \in\mathbb{N}\}$ is also exponentially decreasing because
$\sqrt[n]{b_n} = n^{\alpha/n}q \rightarrow q<1.$
In fact, it can easily be shown that a sequence $\{b_n, \, n \in\mathbb{N}\}$
is exponentially decreasing, that is, \eqref{em2} is satisfied
if and only if there exists a number $q$ ($0<q<1$) such that
\beq
\label{ek4}
b_n =O(q^{n})  \q {\rm as}\q n\to\f.
\eeq
\begin{rem}
{\rm It is easy to see that an exponentially neutral sequence $\{a_n, \, n \in\mathbb{N}\}$
that converges to zero, does so slower than any exponentially decreasing sequence
$\{b_n, \, n \in\mathbb{N}\}$.
In particular, if $\{b_n=q^n, \, 0<q<1, \, n \in\mathbb{N}\}$ is a geometric progression,
then $b_n=o(a_n)$ as $n\to \f$.}
\end{rem}
\begin{rem}
\label{RMT1}
{\rm
It follows from relation \eqref{bb2} that if  $\tau_f:=\tau(\ol E_f)=1$,
then the sequence $\{\sigma_n(f)\}$ is exponentially neutral, and if $\tau_f<1$,
then it is exponentially decreasing.
Thus, Theorem \ref{BB1} shows that the question of exponential decay of the prediction error $\sigma_n(f)$ as $n\to\f$ in fact does not depend on the
values of the spectral density $f(\la)$ on its support $S_f$, and is
determined solely by the value of the transfinite diameter of the closure
of the spectrum $\ol E_f$.
Denote $\g_n:=\si_n(f)/\tau_f^n $. Then
\begin{equation}
\label{sfg}
\si_n(f)=\tau_f^n\cd \g_n,
\end{equation}
and in view of \eqref{bb2} we have
\beq
\label{eng}
\lim_{n \rightarrow \infty} \sqrt[n]{\g_n} =1.
\eeq
Thus, in the case where $\tau_f<1$, the prediction error $\sigma_n(f)$
is decomposed into a product of two factors, one of which ($\tau_f^n$)
is a geometric progression, and the second ($\g_n$) is an exponentially
neutral sequence.
Also, if $g(\la)$ is a spectral density satisfying the conditions of
Theorem \ref{BB1}, then in view of \eqref{sfg}, we have
\begin{equation}
\nonumber
\frac{\si_n(g)}{\si_n(f)}=\left(\frac{\tau_g}{\tau_f}\right)^n\cd\g'_n,
\end{equation}
where $\g'_n$ is an exponentially neutral sequence.}
\end{rem}
The following result contains a sufficient condition for the exponential decay of
$\si_{n}(f)$ as $n\to\f$.
\begin{thm}
\label{cor3.1}
If the spectral density $f(\lambda)$ of the process $X(t)$ vanishes on an interval,
then the prediction error $\si_{n}(f)$ decreases to zero exponentially.
More precisely, if $f(\lambda)$ vanishes on an interval $I_\de\subset [-\pi,\pi]$
of length $2\de$ $(0<\de<\pi)$, then
\begin{equation}
\label{m2}
\limsup_{n \rightarrow \infty} \sqrt[n]{\sigma_n(f)} \leqslant \cos\frac{\de}{2}<1.
\end{equation}
\end{thm}

The next result gives a necessary condition for the exponential decay of $\si_{n}(f)$ as $n\to\f$.
\begin{thm}
\label{cor3.2}
A necessary condition for the prediction error $\si_{n}(f)$ to tend to zero exponentially
is that the spectral density $f(\la)$ should vanish on a set of positive Lebesgue measure, that is, $\mu(S_f^0)>0$, where $S_f^0$ is as in \eqref{S0}.
\end{thm}

\begin{rem}
\label{RMT3}
{\rm
Theorem \ref{cor3.2} shows that if the spectral density $f(\lambda)$
is almost everywhere positive, that is, $\mu(S_f^0)=0$
(in particular, if $S_f^0$ consists of a finite or countable number of points),
then it is impossible to obtain exponential decay of the prediction error $\si_{n}(f)$,
no matter how high the order of the zero of $f(\lambda)$ at the points of $S_f^0$.}
\end{rem}

\sn{Proof of the results of Section \ref{ss3.2}}

We prove here Theorems \ref{BB1} - \ref{cor3.2}.
In Lemma \ref{maz} below, and in what follows, we will use the following
notions and definitions.
A {\it continuum} is defined to be a continuous rectifiable Jordan curve in the complex plane $\mathbb{C}$.
Any subset $E$ of a continuum is called a {\it linear set} in $\mathbb{C}$.
The {\it linear measure} $\mu(E)$ of a linear set $E$ is defined to be the Lebesgue
measure generated by the length of an arc of a continuum (see also Remark \ref{rW}).

The following lemma, which is an immediate consequence of a result by Mazurkievicz
\cite{Maz}, will be used in the proof of Theorem \ref{BB1}.
\begin{lem}
\label{maz}
Let $\G$ be a bounded closed set consisting of a finite number of continua.
Then for any $\epsilon>0$ there is a number $\delta=\delta(\epsilon,\G)>0$
such that for any closed subset $F\subset\G$ and an arbitrary
polynomial $q_n(z)$ of degree $n$ the following inequality holds:
\begin{equation}
\label{3.13}
M_n:=\max_{x \in \G}|q_n(z)|\leq(1+\vs)^n\max_{z\in F}|q_n(z)|,
\end{equation}
provided that $\mu(\G\setminus F)<\delta$.
\end{lem}

In our proof of Theorem \ref{BB1} given below, the set $\G$ will be
either the unit circle $\mathbb{T}$ or the union of a finite number of
closed arcs of $\mathbb{T}$.

\begin{proof}[Proof of Theorem \ref{BB1}]
Define the spectral density:
\beq
\label{tt22}
\ol f(\lambda):= \left \{
\begin{array}{ll}
f(\lambda) & \mbox {if \, $e^{i\lambda}\in E_f$}\\
1 & \mbox {if \, $e^{i\lambda} \in \ol E_f\setminus E_f$}\\
0 & \mbox {if \, $e^{i\lambda} \notin \ol E_f$},
\end{array}
\right.
\eeq
and observe that the spectrum $E_{\ol f}$ of the process with spectral density
$\ol f(\lambda)$ is the closure of the spectrum $E_{f}$ corresponding to $f(\lambda)$
and consists of a finite number of closed arcs:
\beq
\label{SpC}
E_{\ol f}:=\{e^{i\la}: \ol f(\la)> 0\}=\ol E_{f}.
\eeq
Consider the mapping $W$ given by the formula \eqref{W1}
and observe that the functions $\ol f(\lambda)$ and $f(\lambda)$ differ only
on the set $W^{-1}\left(\ol E_{f}\setminus E_{f}\right)=\ol S_{f}\setminus S_{f}$,
which either is empty or consists of a finite number of points.
Hence, in view of Proposition \ref{pp3}(c), we have
\beq
\label{SiC}
\si_n(f)=\si_n(\ol f).
\eeq

We first prove the inequality
\begin{equation}
\label{m3.15}
\limsup_{n \rightarrow \infty} \sqrt[n]{\sigma_n(f)} \leq {\tau}(\ol E_f).
\end{equation}
Denote by $T_n(z, \ol E_f)$ the Chebyshev polynomial of order $n$ of the set
$\ol E_f$, and define (see \eqref{td6})
\beq
\label{mn}
m_n:=m_n(\ol E_f)=\max_{z\in \ol E_f}|T_n(z,\ol E_f)|.
\eeq
Then we can write
\begin{eqnarray}
\nonumber
\sigma_n^2(f) &=&\sigma_n^2(\ol f) = \int_{-\pi}^{\pi} |p_n(e^{i\lambda}, \ol f)|^2
\ol f(\lambda)d\lambda\\
\label{mn3}
&\leq& \int_{-\pi}^{\pi}|T_n(e^{i\lambda}, \ol E_f)|^2 \ol f(\lambda) d\lambda
\leq {m}_n^2(\ol E_f) \cdot \int_{-\pi}^{\pi} \ol f(\lambda)d\lambda.
\end{eqnarray}
The first relation in \eqref{mn3} follows from \eqref{SiC}, the second from
\eqref{mm1}, the third from the definition of optimal polynomial $p_n(z,\ol f)$,
and the fourth from \eqref{mn}.
From \eqref{mn3} we get
\begin{equation}
\label{3.155}
\sigma_n^2(f)\leq c\,{m}_n^2(\ol E_f),
\end{equation}
where $c:=\int_{-\pi}^{\pi}\ol f(\lambda)d\lambda$ is a positive constant.
Taking the root of order $2n$ in \eqref{3.155}, then passing to the limit
as $n\to\f$, in view of \eqref{td7}, \eqref{td11}, Proposition \ref{pp1}(a) and (d),
and the elementary relation $\lim_{n \rightarrow \infty} \sqrt[n]{c} =1$,
we obtain
\begin{equation}
\label{3.15}
\nonumber
\limsup_{n \rightarrow \infty} \sqrt[n]{\sigma_n(f)} \leq {\tau}(\ol E_f)
\leq {\tau}(\mathbb{T})= 1.
\end{equation}

Now we proceed to prove the inequality:
\begin{equation}
\label{3.16}
\liminf_{n \rightarrow \infty} \sqrt[n]{\sigma_n(f)} \geqslant {\tau}(\ol E_f).
\end{equation}
To this end, we consider a sequence of subsets $\{E_n,\, n\in\mathbb{N}\}$ of
$\ol E_f$, defined by
\beq
\label{3.17}
E_n:=\{z \in \ol E_f: \,|p_n(z, \ol f)| > n\sigma_n(\ol f)\},
\eeq
where $\ol f$ is as in \eqref{tt22}, and a measure $\mu_{\ol f}$ on the unit circle
$\mathbb{T}$ generated by the spectral density $\ol f$ as follows:
$$
\mu_{\ol f}(E)=\int_{W^{-1}(E)}\ol f(\lambda)d\lambda, \q E\subset \mathbb{T},
$$
where $W^{-1}(E)$ is as in \eqref{W-1}.
Then, in view of \eqref{Wmm1} and \eqref{3.17} we can write
\begin{eqnarray}
\label{km7}
\nonumber
\sigma_n^2(\ol f) = \int_{\ol E_f} |p_n(z, \ol f)|^2 d\mu_{\ol f}
\geq \int_{E_n} |p_n(z, \ol f)|^2 d\mu_{\ol f} > n^2 \sigma_n^2(\ol f) \mu_{\ol f}(E_n),
\end{eqnarray}
implying that $\mu_{\ol f}(E_n)<n^{-2}$ and
\beq
\label{3.18}
\lim_{n\to\f}\mu_{\ol f}(E_n)=0.
\eeq
Next, since the spectral density $\ol f(\la)$ in \eqref{tt22} is strictly
positive on $W^{-1}(\ol E_f)=\ol S_f$, the Lebesgue measure $\mu_\mathbb{T}$,
defined in \eqref{W2}, is absolutely continuous with respect to the measure
$\mu_{\ol f}$. Hence, taking into account that the measure $\mu_\mathbb{T}$
is also finite ($\mu_\mathbb{T}(\mathbb{T})=2\pi$), by \eqref{3.18}, we have
\begin{equation}
\label{3.19}
\lim_{n \rightarrow \infty} \mu(E_n)  = 0.
\end{equation}

\noindent
Define the sets $F_n: = \ol E_f \backslash E_n$, and observe that $F_n$
are closed subsets of the spectrum $\ol E_f$, and we have (see \eqref{3.17})
\begin{equation}
\label{3.20}
|p_n(z, \ol f)| \leqslant n\sigma_n(\ol f), \q z \in F_n.
\end{equation}
Given an arbitrary $\varepsilon > 0$ we choose $\de:=\de(\ol E_f , \varepsilon)$
according to Lemma \ref{maz} with $\G=\ol E_f$ and $F=F_n$.
Then, in view of \eqref{3.19}, for large enough $n$, we have
$$
\mu(\ol E_f\backslash F_n) = \mu(E_n) < \delta.
$$
Therefore, we can write
\[
{m}_n(\ol E_f) =\max_{z \in \ol E_f} | T_n(z, \ol E_f)| \leqslant \max_{z \in \ol E_f} |p_n(z,
\ol f)| \leq (1 + \varepsilon)^n \max_{z \in F_n} |p_n(z, \ol f)| \leqslant
(1 + \varepsilon)^nn\sigma_n(\ol f),
\]
Here the first and the second relations follow from the definition of
Chebyshev polynomial (see \eqref{td6}), the third from
the relation \eqref{3.13} with $\G=\ol E_f$ and $F=F_n$
(see Lemma \ref{maz}) and the fourth from the inequality \eqref{3.20}.

The last relation implies that
$$ \sigma_n(\ol f)\geqslant \frac{m_n(\ol E_f)}{n(1 + \varepsilon)^{n}}.$$
Taking the root of order $n$, and letting $n$ tend to infinity, in view of the
relation $\lim_{n \rightarrow \infty} \sqrt[n]{n} =1$ and the elementary inequality $(1 + \varepsilon)^{-1}>1 - \varepsilon$, we obtain
\begin{equation}
\label{m15}
\nonumber
\liminf_{n \rightarrow \infty} \sqrt[n]{\sigma_n(\ol f)}
\geqslant {\tau}(\ol E_f) (1 - \varepsilon).
\end{equation}
From the last inequality, taking into account the arbitrariness of $\varepsilon$
and formula \eqref{SiC}, we obtain \eqref{3.16}.
A combination of \eqref{m3.15} and \eqref{3.16} implies \eqref{bb2},
and thus completes the proof of Theorem \ref{BB1}.
\end{proof}

\begin{proof}[Proof of Theorem \ref{cor3.1}]
Denote by $\G_\de:=W(I_\de)$ the arc of the unit circle $\mathbb{T}$ which
is the image of the interval $I_\de$ under the mapping $W$ (see \eqref{W1}),
and let $\G_\al:=\mathbb{T}\setminus \G_\de$ be the complementary of $\G_\de$.
Then $\G_\al$ is a closed arc of the unit circle $\mathbb{T}$, which contains the spectrum
$E_f$ ($E_f\subset\G_\al$) and is of length $2\al$, where $\al=\pi-\de$.
Therefore, by Proposition \ref{pp1}(e) of the transfinite diameter, we have
\beq
\label{t220}
\tau(\G_\al)=\sin\left(\frac{2(\pi-\de)}{4}\right)
=\sin\left(\frac{\pi}{2}-\frac{\de}{2}\right)=\cos\frac{\de}{2}.
\eeq
Next, define the function
\beq
\nonumber
\widehat f(\lambda):= \left \{
\begin{array}{ll}
f(\lambda) & \mbox {if \, $e^{i\lambda}\in E_f$}\\
1 & \mbox {if \, $e^{i\lambda} \in \G_\al\setminus E_f$}\\
0 & \mbox {if \, $e^{i\lambda} \in \G_\de$},
\end{array}
\right.
\eeq
and observe that $E_{\widehat f}=\G_\al$, $f(\lambda)\leq\widehat  f(\lambda)$,
and $\sigma_n(f)\leq \sigma_n(\widehat  f)$ by \eqref{OSm2K}.
Therefore, in view of \eqref{bb2} and \eqref{t220}, we can write
\beq
\label{Gde}
\nonumber
\limsup_{n\to\f}\sqrt[n]{\sigma_n(f)}  \leq \lim_{n\to\f}\sqrt[n]{\sigma_n(\widehat f)}
= \tau(E_{\widehat f})=\tau(\G_\al)=\cos\frac{\de}{2},
\eeq
and the relation \eqref{m2} follows, completing the proof of  Theorem \ref{cor3.1}.
\end{proof}

\begin{proof}[Proof of Theorem \ref{cor3.2}]
%Observe that i
We argue by contradiction. Since by assumption the prediction error $\si_{n}(f)$
decreases to zero exponentially as $n\to\f$, according to Definition\ref{ekd1}(b),
we have
\beq
\label{Gde1}
\limsup_{n \rightarrow \infty} \sqrt[n]{\si_{n}(f)} <1.
\eeq
Assume that the spectral density $f(\la)$ is almost everywhere positive on
$\Lambda=[-\pi,\pi]$, that is, $\mu(S_f^0)=0$, where $S_f^0$ is as in \eqref{S0}. Then, in view of Proposition \ref{pp3}(c), without loss of generality,
we can assume that $S_f=\Lambda$, and hence $\ol E_f=E_f=\mathbb{T}$.
Then according to relation \eqref{bb2} and Proposition \ref{pp1}(d),
it follows that
\beq
\nonumber
\lim_{n\to\f}\sqrt[n]{\sigma_n(f)}  =\tau(\ol E_f)=\tau(\mathbb{T})=1,
\eeq
which contradicts the inequality \eqref{Gde1}, completing the proof
of Theorem \ref{cor3.2}.
\end{proof}

\sn{Some consequences of Theorem \ref{BB1}}
\label{EX}
Motivated by Theorems A and \ref{BB1} and Remark \ref{rm3.5}, the following question
arises naturally: calculate the transfinite diameter $\tau(\ol E_f)$ of closure
of the spectrum $\ol E_f$ consisting of a union of several closed arcs of the unit
circle $\mathbb{T}$, and thus, obtain an explicit
asymptotic relation for the prediction error $\si_n(f)$ similar to the Rosenblatt's
relation \eqref{nd2}.
As it was mentioned in Section \ref{DCC}, the calculation of the transfinite diameter
(and hence, the capacity and the Chebyshev constant) is a challenging problem, and in
only very few cases has the transfinite diameter been exactly calculated (see Proposition \ref{pp1}).
One such example provides Theorem A, in which case the transfinite diameter
of closure of the spectrum  $\ol E_f:=\{e^{i\la}: \la\in [\pi/2-\alpha, \pi/2+\alpha]\}$
is $\sin(\alpha/2)$.
Observe that in \cite{Ros}, M. Rosenblatt calculated the capacity of $\ol E_f$.
Below we give some other examples, where we can explicitly calculate the Chebyshev constant
(and hence the transfinite diameter and the capacity) by using some properties of the transfinite diameter, stated in Proposition \ref{pp1},
and a result due to Robinson \cite{Rob} concerning the relation between the
transfinite diameters of related sets.

In \cite{Rob}, R. Robinson, extending Fekete theorem (see Proposition \ref{pp1}(c)),
proved the following important result about the transfinite diameters of related sets.
\begin{pp}[Robinson \cite{Rob}]
\label{RT2}
Let $F$ be a bounded closed subset of the complex plane $\mathbb{C}$ lying
on the unit circle $\mathbb{T}$ and symmetric with respect to real axis,
and let $F^x$ be the projection of $F$ onto the real axis. Then
\beq
\label{F10}
\tau(F^x)=[2\tau(F)]^{1/2}.
\eeq
\end{pp}
\begin{rem}
\label{ExR}
{\rm The examples given below show that the formula \eqref{F10} gives a
simple way to calculate the transfinite diameters of some
subsets of the circle, based only on the formula \eqref{ts}
of the transfinite diameter of a line segment.}
\end{rem}

We now give examples of calculation of transfinite diameters of some
subsets of the unit circle, using formulas \eqref{ts} and \eqref{F10},
and some properties of the transfinite diameter listed in Proposition \ref{pp1}.

In the examples that follow we will use the following notation:
given $0<\be<2\pi$ and $z_0=e^{i\theta_0}$, $\theta_0\in(-\pi,\pi]$,
we denote by $\G_\be(\theta_0)$
an arc of the unit circle of length $\be$ which is symmetric
with respect to the point $z_0=e^{i\theta_0}$, that is,
\beq
\label{arc}
\G_\be(\theta_0):=\{e^{i\theta}: \, |\theta-\theta_0|\leq\be/2\}
=\{e^{i\theta}: \, \theta\in[\theta_0-\be/2, \theta_0+\be/2]\}.
\eeq
\begin{exa}
\label{ex1}
{\rm Let $\G_{2\al}:=\G_{2\al}(0)$. Then the projection $\G_{2\al}^x$
of $\G_{2\al}$ onto the real axis is the segment $[\cos\al,1]$ (see Figure 1a)),
and by \eqref{ts} for the transfinite diameter $\tau(\G_{2\al}^x)$ we have
$$\tau(\G_{2\al}^x) =\frac{1-\cos\al}4=\frac{\sin^2(\al/2)}2.$$
Hence, according to formula \eqref{F10}, we obtain
\beq
\label{F17}
\tau(\G_{2\al}) =[2\tau(\G_{2\al}^x)]^{1/2}=\left[2\frac{\sin^2(\al/2)}2\right]^{1/2} =\sin\frac\al2.
\eeq
Taking into account that the transfinite diameter is invariant with respect
to rotation (see Proposition \ref{pp1}(b)), from \eqref{F17} for any
$\theta_0\in(-\pi,\pi]$ we have
\beq
\label{F17a}
\tau(\G_{2\al}(\theta_0)) = \sin\frac\al2.
\eeq}
\end{exa}
%\begin{figure}[ht]
%\centering
%\includegraphics[width=5.0in]{Fig-1}
%\caption{a) The set $\G_\al$. b) The set $\G(k)$ with $k=2$.}\label{Figure1}
%\end{figure}
\begin{center}
\includegraphics[height=55mm]{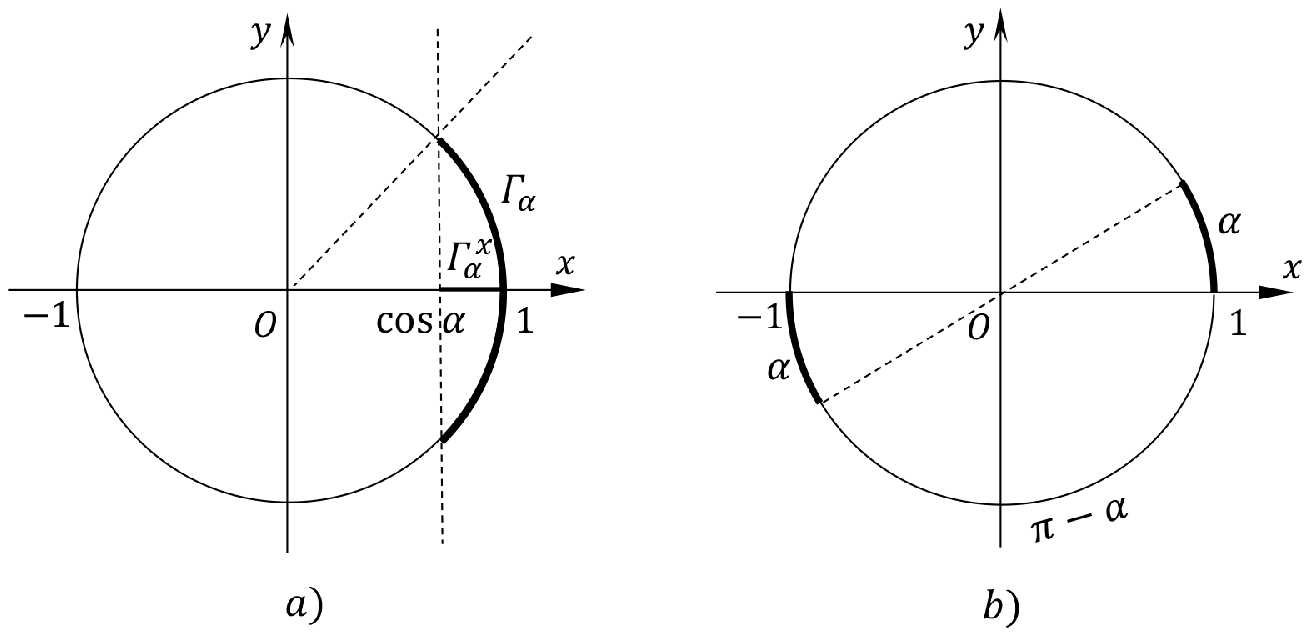}
\end{center}
\centerline{Figure 1. a) The set $\G_\al$. b) The set $\G(k)$ with $k=2$.}

\medskip
\noindent
Notice that formula in \eqref{F17} was first obtained by M. Rosenblatt
in \cite{Ros}, where he calculated the capacity of the arc $\G_{2\al}(\pi/2)$
by using the complex technique of conformal mappings and orthogonal
polynomials on the unit circle.
\begin{exa}
\label{ex2}
{\rm Let $\G_{2\al}(\al)$ be an arc of length $2\al$, defined by \eqref{arc}:
\beq
\label{FF1}
\nonumber
\G_{2\al}(\al)=\{e^{i\theta}: \, \theta\in[0,2\al]\},
\eeq
and let $\G(2)$ be the preimage of the arc $\G_{2\al}(\al)$ under the mapping $p(z)=z^2$. We show that the set $\G(2)$ is the union of two closed arcs of equal lengths $\al$, symmetrically located with respect to the center of the unit circle (see Figure 1b):
\beq
\label{FF2}
\G(2)=\{e^{i\om}: \, \om\in[-\pi,-\pi+\al]\cup [0,\al]\}.
\eeq
Indeed, the preimage $z=e^{i\om}$ of an arbitrary point $e^{i\theta}\in\G_{2\al}(\al)$, $\theta\in[0,2\al]$, under the mapping $p(z)=z^2$ satisfies the equality $z^2=e^{2i\om}=e^{i\theta}$.
This, in view of the $2\pi$-periodicity of $e^{i\theta}$ implies that
$2\om=\theta-2\pi k$, $k\in\mathbb{Z}$, and hence
\beq
\label{FF3}
\om=\om(k)=\frac\theta2-\pi k, \q k\in\mathbb{Z}.
\eeq
Again using the $2\pi$-periodicity of $e^{i\om}$, we conclude that
from the countable set of values of $\om(k)$ in \eqref{FF3} only two values
$\om(0)=\theta/2$ and $\om(1)=\theta/2-\pi$ correspond to distinct
preimages $z=e^{i\om}$ of the point $e^{i\theta}\in\G_{2\al}(\al)$, $\theta\in[0,2\al]$. Thus, each point $e^{i\theta}\in\G_{2\al}(\al)$
has two  distinct preimages $z_1=e^{i\om(0)}=e^{i\theta/2}$ and
$z_2=e^{i\om(1)}=e^{i(\theta/2-\pi)}$. Therefore, for the entire preimage
$\G(2)$ we have
\bea
\nonumber
\G(2)&=&\{e^{i\theta/2}: \, \theta\in[0,2\al]\}\cup
\{e^{i(\theta/2-\pi)}: \, \theta\in [0, 2\al]\}\\
\nonumber
&=&\{e^{i\psi}: \, \psi\in[0,\al]\} \cup \{e^{i\xi}: \, \xi\in [-\pi, -\pi+\al]\}\\
\nonumber
&=&\{e^{i\om}: \, \om\in[-\pi, -\pi+\al]\cup [0,\al]\},
\eea
and \eqref{FF2} follows.
Then, by Fekete theorem (see Proposition \ref{pp1}(c)) and formula \eqref{F17a},
for the transfinite diameter $\tau(\G(2))$ we have
\beq
\label{F188}
\nonumber
\tau(\G(2)) =[\tau(\G_{2\al}(\al))]^{1/2}=\left(\sin\frac{\al}2\right)^{1/2}.
\eeq
The above result can easily be extended to the case of $k$ ($k>2$) arcs.
Let $\G(k)$ be the union of $k$ ($k\in\mathbb{N},\, k\geq2$)
closed arcs of equal lengths $\al$, which are symmetrically located on the
unit circle (the arcs are assumed to be equidistant).
Arguments similar to those applied above can be used to show that the set $\G(k)$
is the preimage (to within rotation) under the mapping $p(z)=z^k$ of the arc $\G_{k\al}(k\al/2)$ of length $k\al$ defined by \eqref{arc}.
Therefore, by Fekete theorem (see Proposition \ref{pp1}(c)) and the invariance
property of the transfinite diameter with respect to rotation (see Proposition \ref{pp1}(b)),
for the transfinite diameter $\tau(\G(k))$, we have
\beq
\label{F18}
\tau(\G(k)) =\left(\sin\frac{k\al}4\right)^{1/k}.
\eeq}
\end{exa}
\begin{exa}
\label{ex3}
{\rm Let $\al>0,$ $\de\geq 0$ and $\al+\de\leq\pi$. Let $\G_{\al,\de}(\theta_0):=\G_{\al+\de}(\theta_0)\setminus\G_{\de}(\theta_0)$
be the union of two arcs of the unit circle of lengths $\al$, the distance of
which (over the circle) is equal to $2\de$. Define (see Figure 2a)):
\beq
\label{F18G}
\G_{\al,\de}:=\G_{\al,\de}(0)=\{e^{i\theta}: \, \theta\in
[-(\de+\al),-\de]\cup [\de,\de+\al]\}.
\eeq
Then the projection $\G_{\al,\de}^x$ of $\G_{\al,\de}$ onto the real axis
is the segment $\G_{\al,\de}^x=[\cos(\al+\de),\cos\de]$, and by \eqref{ts}
for the transfinite diameter $\tau(\G_{\al,\de}^x)$ we have
$$\tau(\G_{\al,\de}^x) =\frac{\cos\de-\cos(\al+\de)}4
=\frac{\sin(\al/2)\sin(\al/2+\de)}2.$$
Hence, according to formula \eqref{F10}, for the transfinite diameter $\tau(\G_{\al,\de})$,
we obtain
\beq
\label{F19}
\tau(\G_{\al,\de})= [2\tau(\G_{\al,\de}^x)]^{1/2} =\left(\sin(\al/2)\sin(\al/2+\de)\right)^{1/2}.
\eeq
In view of Proposition \ref{pp1}(b)), from \eqref{F19} for any
$\theta_0\in(-\pi,\pi]$ we have
\beq
\label{F19a}
\tau(\G_{\al,\de}(\theta_0))=\left(\sin(\al/2)\sin(\al/2+\de)\right)^{1/2}.
\eeq}
\end{exa}
Observe that for $\de=0$ we have $\G_{\al,\de}(\theta_0)=\G_{\al}(\theta_0)$,
and the formula \eqref{F19a} becomes \eqref{F17a}.
\begin{center}
\includegraphics[height=55mm]{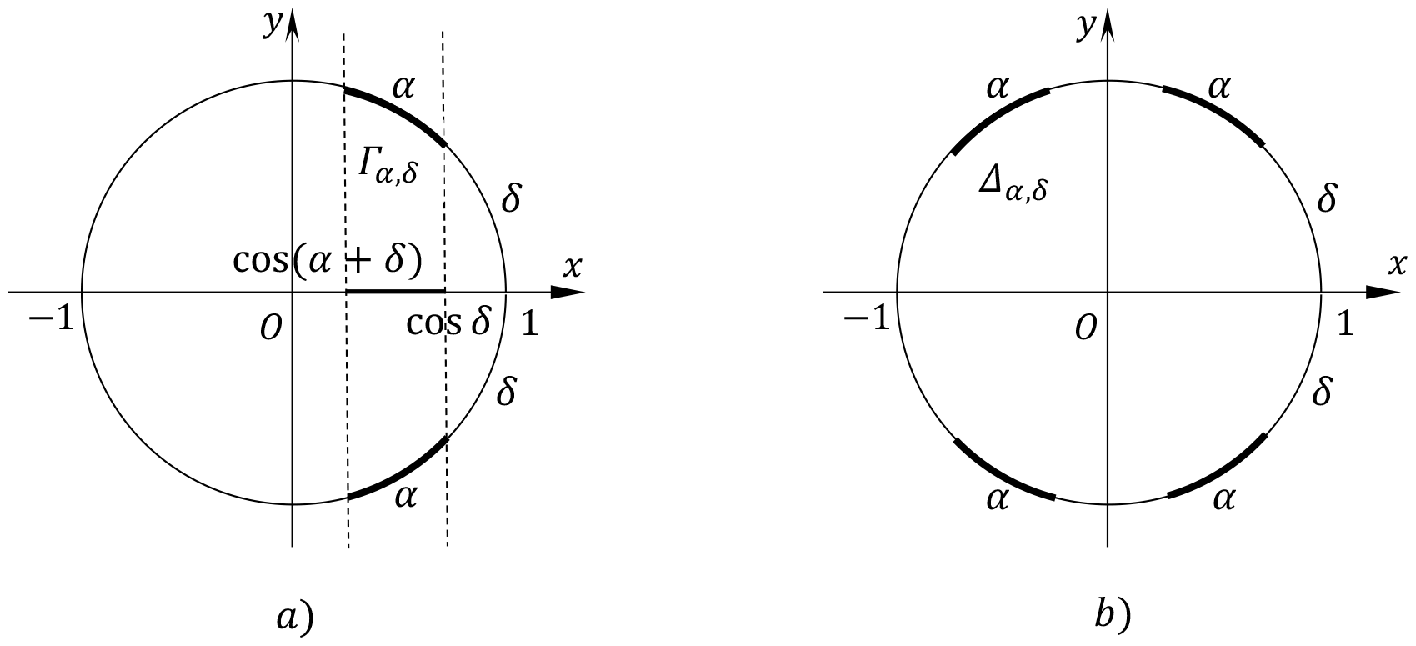}
\end{center}
\centerline{Figure 2. a) The set $\G_{\al,\de}$. b) The set $\De_{\al,\de}$.}
\medskip

%\begin{figure}[ht]
%\centering
%\includegraphics[width=5.0in]{Fig-2}
%\caption{a) The set $\G_{\al,\de}$. b) The set $\De_{\al,\de}$.}\label{Figure2}
%\end{figure}

\begin{exa}
\label{ex4}
{\rm Let the arc $\G_{\al,\de}$ be as in Example \ref{ex3} (see \eqref{F18G})
with $\al,\de$
satisfying $\al+\de\leq\pi/2$, that is, $\G_{\al,\de}$ is a subset of the
right semicircle $\mathbb{T}$. Denote by $\G'_{\al,\de}$ the symmetric to
$\G_{\al,\de}$ set with respect to $y$-axis, that is,
$$\G'_{\al,\de}:=\{e^{i\theta}: \, \theta\in
[-\pi+\de,-\pi+(\de+\al)] \cup [\pi-(\de+\al),\pi-\de]\}.$$
Define
$\Delta_{\al,\de}:=\G_{\al,\de}\cup\G'_{\al,\de}$,
and observe that the set $\Delta_{\al,\de}$ consists of four arcs of equal lengths $\al$,
which are symmetrically located with respect to both axes (see Figure 2b)).
Arguments similar to those applied in Example \ref{ex2} can be used to show that the set $\Delta_{\al,\de}$ is the preimage (to within rotation) of the set $\G_{2\al,2\de}$ under the mapping $p(z)=z^2$.
Hence, according to Fekete theorem (see Proposition \ref{pp1}(c)) and
formula \eqref{F19}, for the transfinite diameter
$\tau(\Delta_{\al,\de})$, we obtain
\beq
\label{F199}
\tau(\Delta_{\al,\de})= \left(\tau(\G_{2\al,2\de}\right)^{1/2}
=\left(\sin\al\sin(\al+2\de)\right)^{1/4}.
\eeq
Denote by $\Delta_{\al,\de}(\theta_0)$ the image of the set $\Delta_{\al,\de}$
under mapping $q(z)=e^{i\theta_0}z$, that is, under the rotation by the angle
$\theta_0$ around the origin. Then, in view of Proposition \ref{pp1}(b)),
from \eqref{F199} for any $\theta_0\in(-\pi,\pi]$ we have
\beq
\label{F199a}
\tau(\Delta_{\al,\de}(\theta_0))=\left(\sin\al\sin(\al+2\de)\right)^{1/4}.
\eeq
}
\end{exa}

Now we apply Theorem \ref{BB1} to obtain
the asymptotic behavior of the prediction error $\si_n(f)$ for some specific spectra.
Putting together Theorem \ref{BB1} and Examples \ref{ex1}-\ref{ex4},
we can state the following result.

\begin{thm}
Let $\ol E_f$ be the closure of the spectrum $E_f$ of a stationary process $X(t)$.
The following assertions hold.
\begin{itemize}
\item[{\rm(a)}] If $\ol E_f= \G_{2\al}(\theta_0)$, where $\G_{2\al}(\theta_0)$,
    $\theta_0\in (-\pi,\pi]$, is as in Example
\ref{ex1}, then
\beq
\label{F30}
\nonumber
\lim_{n\to\f}\sqrt[n]{\sigma_n(f)}  %=\tau(E_f)=\tau(\G_\al)
=\sin\frac\al2.
\eeq
\item[{\rm(b)}] If $\ol E_f= \G(k)$, where $\G(k)$ is as in Example \ref{ex2}, then
\beq
\label{F31}
\nonumber
\lim_{n\to\f}\sqrt[n]{\sigma_n(f)}  %=\tau(E_f)=\tau(\G)
=\left(\sin\frac{k\al}4\right)^{1/k}.
\eeq
\item[{\rm(c)}] If $\ol E_f= \G_{\al,\de}(\theta_0)$, $\theta_0\in (-\pi,\pi]$, where $\G_{\al,\de}(\theta_0)$
is as in Example \ref{ex3}, then
\beq
\label{F32}
\nonumber
\lim_{n\to\f}\sqrt[n]{\sigma_n(f)} % =\tau(E_f)=\tau(\G_{\al,\de})
=\left(\sin(\al/2)\sin(\al/2+\de)\right)^{1/2}.
\eeq
\item[{\rm(d)}] If $\ol E_f= \Delta_{\al,\de}(\theta_0)$, where $\Delta_{\al,\de}(\theta_0)$, $\theta_0\in (-\pi,\pi]$,
is as in Example \ref{ex4}, then
\beq
\label{F33}
\nonumber
\lim_{n\to\f}\sqrt[n]{\sigma_n(f)}  %=\tau(E_f)=\tau(\Delta_{\al,\de})
=\left(\sin\al\sin(\al+2\de)\right)^{1/4}.
\eeq
\end{itemize}
\end{thm}
\begin{rem}
{\rm The assertion (a) is a slight extension of the Rosenblatt's relation \eqref{nd2}.
The assertion (c) is an extension of assertion (a), which reduces to assertion (a) if $\de=0$.}
\end{rem}

\section{An extension of Rosenblatt's second theorem}
\label{d2}

\sn{Preliminaries}
\label{pre}
In this subsection, we analyze the asymptotic behavior of the prediction error
in the case where the spectral density $f(\la)$ of the model has a very high order
contact with zero at one or several points, so that the Szeg\"o condition \eqref{S} is violated.

Based on Rosenblatt's result for this case, namely Theorem B,
we can expect that for any deterministic process with spectral density possessing
a singularity of the type \eqref{nd5}, the rate of the prediction error $\si^2_n(f)$
should be the same as in \eqref{nd6}. However, the method applied in Rosenblatt \cite{Ros}
does not allow to prove this assertion. Here, using a different approach,
we extend Rosenblatt's second theorem to a broader class of spectral densities.

To state the corresponding results we need some definitions and preliminaries.
We first introduce the notion of weakly varying sequences and state some of
their properties.
\begin{den}
\label{kd1}
A sequence of non-zero numbers $\{a_n, \, n \in\mathbb{N}\}$ is said to be
weakly varying if
\beq
\label{k4}
\nonumber
\lim_{n\to\f} \frac{a_{n+1}}{a_{n}} =1 .
\eeq
\end{den}
For example, the sequence $\{n^{\al}, \, \, \al\in\mathbb{R}, \, n\in\mathbb{N}\}$
is weakly varying (for $\al<0$ it is weakly decreasing and for $\al>0$
it is weakly increasing), while the geometric progression $\{q^n, \, 0<q<1, \,
n \in\mathbb{N}\}$ is not weakly varying.

In the next proposition we list some simple properties of the weakly varying
sequences, which can easily be verified.
\begin{pp}
\label{p4.1}
The following assertions hold.
\begin{itemize}
\item[(a)]
If $\{a_n, \, n \in\mathbb{N}\}$ is a weakly varying sequence, then
for any $\nu\in \mathbb{N}$
\beq
\label{k41}
\lim_{n\to\f} \frac{a_{n+\nu}}{a_{n}} =1.
\eeq
\item[(b)]
If $\{a_n, \, n \in\mathbb{N}\}$ is a sequence such that $a_n\to a\neq0$
as $n\to\f$, then $\{a_n\}$ is a weakly varying sequence.
\item[(c)]
If $\{a_n, \, n \in\mathbb{N}\}$ and $\{b_n, \, n \in\mathbb{N}, \}$ are
weakly varying sequences, then $ca_n$ $(c\neq0),$
$a_n^\al \, (\al\in\mathbb{R})$, $a_nb_n$ and $a_n/b_n$ also are
weakly varying sequences.
\item[(d)]
If $\{a_n, \, n \in\mathbb{N}\}$ is a weakly varying sequence,
and $\{b_n, \, n \in\mathbb{N}\}$ is a sequence of non-zero numbers such that
\beq
\label{k5}
\lim_{n\to\infty}\frac{b_n}{a_n}=c\neq 0,
\eeq
then $\{b_n, \, n \in\mathbb{N}\}$ is also a weakly varying sequence.
\item[(e)]
If $\{a_n, \, n \in\mathbb{N}\}$ is a weakly varying sequence of
positive numbers, then it is exponentially neutral, that is,
it satisfies the relation \eqref{en} (see Definition \ref{ekd1}).
\end{itemize}
\end{pp}
\begin{rem}
\label{km4.2}
{\rm
Proposition \ref{p4.1}(b) implies that if a convergent sequence
has a non-zero limit, then it is weakly varying.
However, a sequence that converges to zero need not necessarily
be weakly varying. For instance, the geometric progression
$\{q^n, \, 0<q<1, \, n \in\mathbb{N}\}$ tends to zero,
but it is not a weakly varying sequence.
Thus, the notion of weakly varying sequence is actual and
non-trivial only if the sequence converges to zero.}
\end{rem}
In the next definition we introduce certain classes of bounded functions.
\begin{den}
\label{kd2}
We define the class $B$ to be the set of all nonnegative, Riemann integrable on on $\Lambda=[-\pi,\pi]$ functions $h(\lambda)$. Also, we define the following subclasses:
\beq
\label{A}
B_+:= \{h \in B: \, h(\lambda)\geqslant m>0\},\q
B^-:= \{h \in B: \, h(\lambda)\leqslant M<\infty\}, \q B_+^-:=B_+\cap B^-.
\eeq
\end{den}

Recall that a trigonometric polynomial $t(\lambda)$ of degree $\nu$ is a
function of the form:
\beq
\label{tp1}
\nonumber
t(\lambda)= a_0+\sum_{k=1}^\nu(a_k\cos k\la + b_k\sin k\la)
= \sum_{k=-\nu}^\nu c_ke^{ik\la}, \q \la\in \mathbb{R},
\eeq
where $a_0, a_k, b_k \in \mathbb{R}$, $c_0=a_0$, $c_k=1/2(a_k-ib_k)$,
$c_{-k}=\ol c_k=1/2(a_k+ib_k)$, $k=1,2,\ldots, \nu$.

In the next proposition we list some properties of the geometric mean
of a function (see formula \eqref{a2}) and trigonometric polynomials.
\begin{pp}
\label{p4.2}
The following assertions hold.
\begin{itemize}
\item[(a)]
Let $c>0$, $\al\in\mathbb{R}$, $f(\la)\geq0$ and $g(\la)\geq0$. Then
\beq
\label{gt1}
G(c)=c, \q G(fg)=G(f)G(g), \q G(f^\al)=G^\al(f).
\eeq
\item[(b)]
{\rm (Fej\'er-Riesz theorem)}. Let $t(\lambda)$ be a nonnegative trigonometric polynomial of degree $\nu$. Then
there exists an algebraic polynomial  $s_{\nu}(z)$ $(z\in \mathbb{C})$
of same degree $\nu$, such that $s_{\nu}(z)\neq0$ for $|z|<1$, and
\beq
\label{ss2}
t(\lambda)=|s_{\nu}(e^{i\lambda})|^2.
\eeq
Under the additional condition $s_{\nu}(0)>0$ the polynomial $s_{\nu}(z)$
is determined uniquely.
\item[(c)]
Let $t(\lambda)$ and  $s_{\nu}(z)$ be as in Assertion (b). Then
\beq
\label{k9}
G(t)=|s_{\nu}(0)|^2>0,
\eeq
where $G(t)$ is the geometric mean of $t(\la)$.
\end{itemize}
\end{pp}
\begin{proof}
Assertion (a) immediately follows from the definition
of the geometric mean (see formula \eqref{a2}) and the properties of
exponent and logarithm. The proof of Assertion (b) (Fej\'er-Riesz theorem) can be found, for example, in Grenander and Szeg\"o \cite{GS}, Section 1.12, p. 20-22.
Assertion (c) follows from Assertion (b). Indeed,  observing that
$\ln |s_\nu(z)|^2$ is a harmonic function, by the well-known mean-value
theorem for harmonic functions (see, e.g., Ahlfors \cite{Ah}, p.165), we have
$$\ln |s_\nu(0)|^2=\frac1{2\pi}\int_{-\pi}^{\pi}\ln |s_\nu(e^{i\lambda})|^2d\lambda=\frac1{2\pi}\int_{-\pi}^{\pi}\ln t(\lambda)d\lambda=\ln G(t),$$
and \eqref{k9} follows.
\end{proof}

In what follows we consider the class of {\it deterministic} processes for
which the sequence of
prediction errors $\{\si_n(f)\}$ is weakly varying, that is,
\beq
\label{k04}
\lim_{n\to\f} \frac{\sigma_{n+1}(f)}{\sigma_{n}(f)} =1.
\eeq
\begin{rem}
\label{km4.25}
{\rm In view of relation \eqref{k04} and Theorems \ref{BB1}-\ref{cor3.2},
%and Remark \ref{km4.2},
for the processes from the class described above the spectral density $f(\la)$ can vanish only on a "rare" set of points, and the Szeg\"o condition \eqref{S} is now violated
due to a very high order contact with zero for at least one point
of the set $S_f^0$, where $S_f^0$ is as in \eqref{S0}.}
\end{rem}

\sn{An extension of Rosenblatt's second theorem}
We first examine the asymptotic behavior as $n\to\f$ of the ratio:
\beq
\label{k1}
\frac{\sigma_n^2(fg)}{\sigma_n^2(f)},
\eeq
where $g(\lambda)$ is some nonnegative function, such that  $fg\in L^1(\Lambda)$.

To clarify the approach, we first assume that $f(\lambda)$ is a spectral
density of a nondeterministic process, in which case the geometric mean $G(f)$
is positive (see \eqref{c013} and \eqref{a2}).
We can then write
\beq
\label{k3}
\lim_{n\to\infty}\frac{\sigma_n^2(fg)}{\sigma_n^2(f)}
=\frac{\sigma_\infty^2(fg)}{\sigma_\infty^2(f)}=\frac{2\pi G(fg)}{2\pi G(f)}
=\frac{G(f)G(g)}{G(f)}=G(g).
\eeq
It turns out that under some additional assumptions imposed on functions
$f$ and $g$, the asymptotic relation (\ref{k3}) remains also valid in the
case of deterministic processes, that is, when $\si^2_\f(f)=0$, or equivalently, $G(f)=0$.

We are now in position to state the main results of this section.\\
The following theorem describes the asymptotic behavior of the ratio \eqref{k1}
as $n\to\f$ for the class of processes described above, and essentially states
that if the spectral density $f(\la)$ is such that the sequence $\{\si_n(f)\}$
is weakly varying, and $g(\la)$ is the spectral density of a nondeterministic
process satisfying some conditions, then the sequences $\{\si_n(fg)\}$ and
$\{\si_n(f)\}$  have the same asymptotic behavior as $n\to\f$.
\begin{thm}
\label{sT1}
Suppose that $f(\la)$ is the spectral density of a deterministic
process such that the sequence $\{\si_n(f)\}$ is weakly varying, that is, the condition \eqref{k04} is satisfied. Let $g(\la)$ be a function of the form:
\beq
\label{g}
g(\la)=h(\la)\cd\frac{t_1(\la)}{t_2(\la)},
\eeq
where  $h(\la)\in B_+^-$ and $t_1(\la)$,
$t_2(\la)$ are nonnegative trigonometric polynomials, such that $f(\la)g(\la)\in A$. Then $g(\la)$ is the spectral density of a
nondeterministic process and the following relation holds:
\beq
\label{k7}
\lim_{n\to\f}\frac{\si^2_{n}(fg)}{\si^2_{n}(f)} =G(g)>0,
\eeq
where $G(g)$ is the geometric mean of $g(\la)$.
\end{thm}
As an immediate consequence of Theorem \ref{sT1} and Proposition \ref{p4.1}(d),
we have the following result.
\begin{cor}
\label{c4.1}
Let the spectral densities $f(\la)$ and $g(\lambda)$ be as in Theorem \ref{sT1}.
Then the sequence $\sigma_n(fg)$ is also weakly varying.
\end{cor}
Taking into account that the sequence $\{n^{-\al}, \, \, n\in\mathbb{N}, \, \al>0\}$
is weakly varying, as an immediate consequence of Theorem \ref{sT1}, we obtain the following result.
\begin{cor}
\label{c4.2}
Let the spectral densities $f(\la)$ and $g(\lambda)$ be as in Theorem \ref{sT1},
and let $\sigma_n(f)\sim cn^{-\al}$ ($c>0, \al>0$) as $n\to\f$. Then
\beq
\label{s88}
\nonumber
\sigma_n(fg)\sim c G(g)n^{-\al} \q {\rm as} \q n\to\f,
\eeq
where $G(g)$ is the geometric mean of $g(\la)$.
\end{cor}
\label{c4.3}
The next result, which immediately follows from Theorem B and Corollary \ref{c4.2},
extends Rosenblatt's Theorem B.
\begin{thm}
\label{sT2}
Let $f(\la)=f_a(\la)g(\la)$, where $f_a(\la)$ is defined by (\ref{nd4})
and $g(\la)$ satisfies the assumptions of Theorem \ref{sT1}. Then
\beq
\label{s8}
\nonumber
\de_n(f) = \si^2_n(f)\sim\frac{\Gamma^2\left(\frac{a+1}2\right)G(g)}
{\pi 2^{2-a}} \ n^{-a} %\sim n^{-a}
\q {\rm as} \q n\to\f,
\eeq
where $G(g)$ is the geometric mean of $g(\la)$.
\end{thm}
We thus obtain the same limiting behavior for $\si^2_n(f)$ as in
the Rosenblatt's relation \eqref{nd6} up to an additional
positive factor $G(g)$.

\sn{Auxiliary lemmas}

To prove Theorem \ref{sT1}, we first establish a number of lemmas.

\begin{lem}
\label{kl2}
Assume that the sequence $\sigma_n(f)$ is weakly varying, that is,
it satisfies the condition \eqref{k04}.
Then for any nonnegative trigonometric polynomial $t(\lambda)$ we have
\beq
\label{k8}
\liminf_{n\to\infty}\frac{\sigma_n^2(ft)}{\sigma_n^2(f)}\geqslant G(t)>0,
\eeq
where $G(t)$ is the geometric mean of $t(\la)$.
\end{lem}
\begin{proof}
Let the polynomial $t(\lambda)$ be of degree $\nu$, and let $s_\nu(z)$
be the algebraic polynomial of degree $\nu$ from the Fej\'er-Riesz
representation \eqref{ss2}.

Let $p^*_n(z,ft)$ be the optimal polynomial of degree $n$ for $f(\la)t(\la)$
in the class $\mathcal{Q}^*_n$ (see formula \eqref{mmm1}).
We now introduce a new polynomial:
\beq
\label{k98}
r_{n+\nu}(z):=p^*_n(z,ft)\frac{s_\nu(z)}{s_\nu(0)},
\eeq
and observe that $r_{n+\nu}(z)\in\mathcal{Q}^*_{n+\nu}$. Therefore
\beq
\label{k99}
\int_{-\pi}^{\pi}|r_{n+\nu}(e^{i\lambda})|^2f(\lambda)d\lambda
\geqslant \int_{-\pi}^{\pi}|p^*_{n+\nu}(e^{i\lambda},f)|^2f(\lambda)d\lambda,
\eeq
where $p^*_{n+\nu}(z,f)$ is the optimal polynomial of degree $n+\nu$
for $f(\la)$ in the class $\mathcal{Q}^*_n$.

Next, we can write
$$\sigma_n^2(ft)=\int_{-\pi}^{\pi}|p^*_n(e^{i\lambda},ft)|^2f(\lambda)t(\lambda)d\lambda
=\int_{-\pi}^{\pi}|p^*_n(e^{i\lambda},ft)s_\nu(e^{i\lambda})|^2f(\lambda)d\lambda$$
$$=|s_\nu(0)|^2\int_{-\pi}^{\pi}|r_{n+\nu}(e^{i\lambda})|^2f(\lambda)d\lambda
\geqslant |s_\nu(0)|^2\int_{-\pi}^{\pi}|p^*_{n+\nu}(e^{i\lambda},f)|^2f(\lambda)d\lambda
=|s_\nu(0)|^2\si^2_{n+\nu}(f).$$
Here the first relation follows from formula \eqref{mmm1}, the second from
Fej\'er-Riesz representation \eqref{ss2}, the third from  \eqref{k98},
the fourth from  \eqref{k99}, and the fifth from \eqref{mmm1}.
Therefore, in view of \eqref{k9}, we obtain
\beq
\label{k10}
\liminf_{n\to\infty}\frac{\sigma_n^2(ft)}{\sigma_{n+v}^2(f)}\geqslant|s_\nu(0)|^2=G(t).
\eeq
Now, taking into account (\ref{k04}) and Proposition \ref{p4.1}(a),
from (\ref{k10}) we obtain (\ref{k8}).
\end{proof}

\begin{lem}
\label{kl3}
Let the sequence $\sigma_n(f)$ satisfy (\ref{k04}), and let $t(\lambda)$ be a
nonnegative trigonometric polynomial such that the function $f(\lambda)/t(\lambda)\in B$.
Then the following inequality holds:
\beq
\label{k11}
\limsup_{n\to\infty}\frac{\sigma_n^2(f/t)}{\sigma_n^2(f)}\leqslant G(1/t),
\eeq
where $G(1/t)$ is the geometric mean of $1/t(\la)$ and $G(1/t)>0$.
\end{lem}

\begin{proof}
Let $s_\nu(z)$ be the algebraic polynomial of degree $\nu$ from the Fej\'er-Riesz
representation \eqref{ss2} for polynomial $t(\la)$, and
let $p^*_n(z,f/t)$ be the optimal polynomial of degree $n$ for $f(\la)/t(\la)$
in the class $\mathcal{Q}^*_n$ (see formula \eqref{mmm1}). For $n>\nu$ we set
$$
r_n(z):=p^*_{n-\nu}(z,f)\frac{s_\nu(z)}{s_\nu(0)},$$
and observe that $r_{n}(z)\in\mathcal{Q}^*_{n}$. Therefore, we have
\beaa
\sigma_n^2(f/t)&&=\int_{-\pi}^{\pi}|p^*_n(e^{i\lambda},f/t)|^2f(\lambda)/t(\lambda)d\lambda
\leq\int_{-\pi}^{\pi}|r_n(e^{i\lambda})|^2f(\lambda)/t(\lambda)d\lambda\\
&&=\frac1{|s_\nu(0)|^2}\int_{-\pi}^{\pi}|p^*_{n-\nu}(e^{i\lambda},f)|^2f(\lambda)d\lambda
=\frac1{|s_\nu(0)|^2}\si^2_{n-\nu}(f),
\eeaa
which, in view of  \eqref{k9} and \eqref{gt1}, implies that
\beq
\label{k121}
\limsup_{n\to\infty}\frac{\sigma_n^2(f/t)}{\sigma_{n-\nu}^2(f)}
\leqslant\frac1{|s_\nu(0)|^2}= G(1/t).
\eeq
Finally, taking into account (\ref{k04}) and Proposition \ref{p4.1}(a),
from (\ref{k121})
we obtain (\ref{k11}).
\end{proof}

In the next lemma we approximate in the space $L_1$ a function
from the class $B_+^-$ by a trigonometric polynomial with special features.
\begin{lem}
\label{kl4a}
Let $h(\lambda)$ be a function from the class $B_+^-$.
Then for any $\vs>0$ a trigonometric polynomial $t(\lambda)$
can be found to satisfy the following condition:
\bea
\label{k18a}
\|h-t\|_1:=\int_{-\pi}^{\pi}|h(\lambda)-t(\lambda)|d\lambda\leqslant\epsilon.
\eea
Moreover, if $m$ and $M$ are the constants from the Definition \ref{kd2}
(see \eqref{A}),
then the polynomial $t(\lambda)$ can be chosen so that for all $\la\in[-\pi,\pi]$
one of the following relations is satisfied:
\bea
\label{k18b}
m-\vs<t(\lambda)< h(\lambda),
\eea
\bea
\label{k18cd}
h(\lambda)<t(\lambda)<M+\vs.
\eea
\end{lem}
\begin{proof}
We first prove the combination of inequalities \eqref{k18a} and \eqref{k18b}.

Let $\{\la_i\}$ ($-\pi=\la_0<\la_1<\cdots <\la_k=\pi$) be a partition
of the segment $[-\pi,\pi]$, and let
$s$ be the Darboux lower sum corresponding to this partition:
$$
s=\sum_{i=1}^km_i\Delta \la_i, \,\, m_i=\inf_{\la\in\Delta_i}h(\la), \,\,
\Delta_i =[\la_{i-1},\la_{i}], \,\, \Delta \la_i =\la_{i}-\la_{i-1}, \,\,
i=1,\ldots, k.
$$
On the segment $[-\pi,\pi]$ we define a step-function $\I_k(\la)$
corresponding to given partition as follows (see Figure 3):
\beq
\label{aa2}
\nonumber
\I_k(\la): = \left \{
\begin{array}{lll}
m_i, & \mbox{if \, $\la\in (\la_{i-1},\la_{i}), \,\, i=1,\ldots, k-1,$}\\
\min\{m_{i}, m_{i+1}\}, & \mbox {if $\la=\la_i,$}\\
\min\{m_{1}, m_{k}\}, & \mbox {if $\la=\la_0$ or $\la=\la_k$.}
%m_k, & \mbox {if $\la=\la_k.$}
           \end{array}
           \right.
\eeq
%\begin{figure}[ht]
%\centering
%\includegraphics[width=5.0in]{Figure-3.eps}
%\caption{{\color{blue}$-\hspace{-2mm}-\hspace{-2mm}-\hspace{-2mm}-$}
%graph of the function $h(\la)$;
%{\color{red}$\bullet\hspace{-2mm}\longrightarrow$}
%graph of the function $\I_k(\la)$.}\label{Figure3}
%\end{figure}
\begin{center}
\includegraphics[height=65mm]{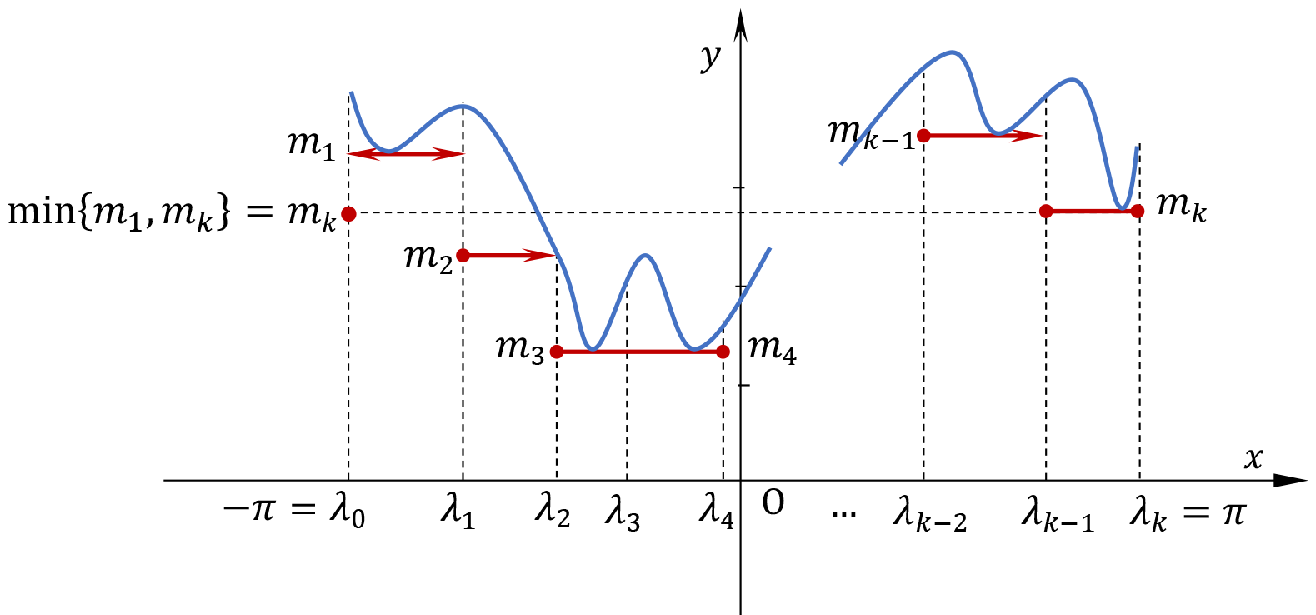}
\end{center}
\centerline{Figure 3. {\color{blue}$-\hspace{-2mm}-\hspace{-2mm}-\hspace{-2mm}-$}
graph of the function $h(\la)$; {\color{red}$\bullet\hspace{-2mm}\longrightarrow$}
graph of the function $\I_k(\la)$.}

\medskip

Observe that if $m_1=m_k$, then the steps of the function $\I_k(\la)$ are intervals,
and the number of steps is equal to $k$. In the case where
$m_1\neq m_k$, one more step (the fist or the last) is added, which
consists of one point (on Figure 3, this is the first step, consisting
of the point with coordinates $(\la_0, m_k)$).

It is clear that the function $\I_k(\la)$ satisfies the following conditions:
\bea
\label{k18c}
%\nonumber
\I_k(\la)\leq h(\lambda), \, \, \lambda\in [-\pi,\pi] \q {\rm and} \q
\int_{-\pi}^\pi\I_k(\la)d\lambda=s.
\eea
Since the function $h(\la)$ is integrable, for an arbitrary given $\epsilon>0$
a partition of the segment $[-\pi,\pi]$ can be found so that the corresponding
Darboux lower sum satisfies the condition:
\bea
\label{k18d}
\int_{-\pi}^\pi h(\la)d\lambda-s= \int_{-\pi}^\pi[h(\la)-\I_k(\la)]d\lambda
=\|h-\I_k\|_1< \frac{\epsilon}3.
\eea

Now using the function $\I_k(\la)$ we construct a new  function that is
continuous  on $[-\pi,\pi]$.
To this end, we connect all the adjacent steps of the graph of $\I_k(\la)$
by slanting line segments as follows:
for each partition point $\la_i, i=1,\ldots, k-1$, at which the
function $\I_k(\la)$ is discontinuous, the endpoint of the lower step
of the graph with abscissa $\la_i$ we connect by a line segment with some
interior point of the adjacent upper step, with the abscissa $\la_i^*$
satisfying the condition (see Figure 4):
\bea
\label{k18m}
|\la_i-\la_i^*|<\vs/(3kM).
\eea
Then, we remove the part of the upper step lying under the constructed slanting segment.
The obtained polygonal line is a graph of some continuous piecewise linear function,
which we denote by $h_k(\la)$. According to the construction and \eqref{k18c},
this function satisfies the condition:
\bea
\label{k18e}
%m\leq
h_k(\la)\leq\I_k(\la)\leq h(\lambda)\leq M, \q \lambda\in [-\pi,\pi].
\eea
%\begin{figure}[ht]
%\centering
%\includegraphics[width=5.0in]{Figure-4.eps}
%\caption{{\color{blue}$-\hspace{-2mm}-\hspace{-2mm}-\hspace{-2mm}-$}
%graph of the function $h(\la)$;
%{\color{red}$\bullet\hspace{-2mm}-\hspace{-2mm}-\hspace{-2mm}-\hspace{-5mm}-\bullet$}
%graph of the function $h_k(\la)$.}\label{Figure4}
%\end{figure}

\begin{center}
\includegraphics[height=65mm]{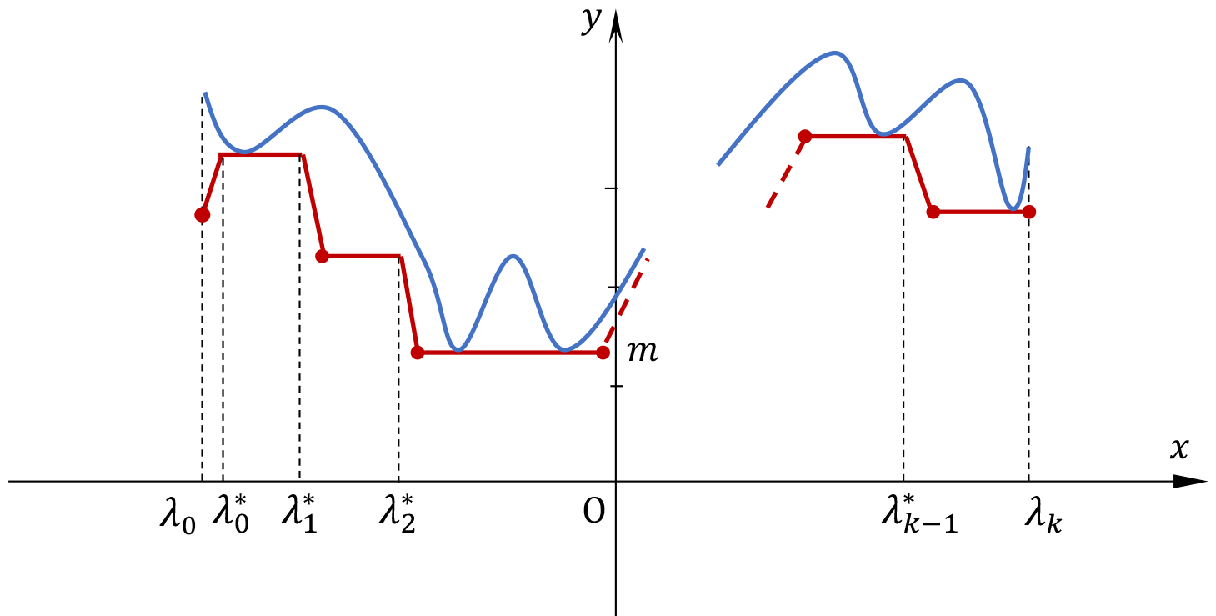}
\end{center}
\centerline{Figure 4.
{\color{blue}$-\hspace{-2mm}-\hspace{-2mm}-\hspace{-2mm}-$} graph of the function $h(\la)$; {\color{red}$\bullet\hspace{-2mm}-\hspace{-2mm}-\hspace{-2mm}-\hspace{-5mm}-\bullet$}
graph of the function $h_k(\la)$.}

\medskip

Taking into account that the functions $h_k(\la)$ and $\I_k(\la)$
coincide outside the segments $[\la_i, \la_i^*]$ (or $[\la_i^*, \la_i]$),
in view of \eqref{k18m} and \eqref{k18e}, we can write
\bea
\label{k18f}
\|\I_k-h_k\|_1= \int_{-\pi}^\pi[\I_k(\la)-h_k(\la)]d\lambda
=\sum_{i=1}^{k-1}\left|\int_{\la_i}^{\la_i^*}[\I_k(\la)-h_k(\la)]d\lambda\right|
< \frac{\epsilon}3.
\eea
Notice that the function $h_k(\la)$ is continuous on the segment $[-\pi,\pi]$
and satisfies the condition $h_k(-\pi) =h_k(\pi)$. Hence, according to Weierstrass
theorem (see, e.g., Grenander and Szeg\"o \cite{GS}, Section 1.9, p. 15),
a trigonometric polynomial $\widetilde t(\lambda)$
can be found so that uniformly for all $\lambda\in [-\pi,\pi]$,
\bea
\label{k18g}
-\frac{\epsilon}{12\pi}<h_k(\la)-\widetilde t(\la)< \frac{\epsilon}{12\pi}.
\eea
Setting $t(\la):= \widetilde t(\la)- \frac{\epsilon}{12\pi}$, from \eqref{k18g} we get
\bea
\label{k18mg}
0<h_k(\la)- t(\la)< \frac{\epsilon}{6\pi}.
\eea
Therefore
\bea
\label{k18i}
\|h_k-t\|_1= \int_{-\pi}^\pi[h_k(\la)-t(\la)]d\lambda < \frac{\epsilon}3.
\eea
Combining the inequalities \eqref{k18d}, \eqref{k18f} and \eqref{k18i},
we obtain
$$
\|h-t\|_1\leq \|h-\I_k\|_1+\|\I_k-h_k\|_1+\|h_k-t\|_1\leq  \epsilon,
$$
and the inequality \eqref{k18a} follows.

Now we proceed to prove the inequality \eqref{k18b}.
Observe first that the second inequality in \eqref{k18b} follows
from the first inequality in  \eqref{k18mg} and \eqref{k18e}.
To prove the first inequality in \eqref{k18b}, observe that by the construction of
the function $h_k(\la)$, we have
\bea
\label{k18x}
h_k(\la)\geq\min\{m_1,\ldots,m_k\}\geq m.
\eea
Next, in view of the second inequality in \eqref{k18mg}, we get
\bea
\label{k19x}
t(\la)\geq h_k(\la)- \frac{\epsilon}{6\pi}>h_k(\la)- \epsilon.
\eea
Combining \eqref{k18x} and \eqref{k19x}, we obtain the first inequality in \eqref{k18b}.

The combination of inequalities \eqref{k18a} and \eqref{k18cd} can be
proved similarly with the following changes:
instead of Darboux lower sum the upper sum should be used,
in the definition of function $\I_k(\la)$ instead of minima should be
taken maxima,
and in the construction of function $h_k(\la)$, the endpoint of the
upper step of the function $\I_k(\la)$ should be connected with an
interior point of the adjacent lower step.
%Lemma \ref{kl4a} is proved.
\end{proof}

\begin{lem}
\label{kl4}
Let $h(\lambda)\in B_+^-$ and let the sequence $\sigma_n(f)$ satisfy (\ref{k04}).
Then the following asymptotic relation holds:
\beq
\label{k12}
\lim_{n\to\infty}\frac{\sigma_n^2(fh)}{\sigma_n^2(f)}=G(h)>0.
\eeq
\end{lem}

\begin{proof}
Observe first that together with $h(\la)$ the function $1/h(\la)$ also
belongs to the class $B_+^-$:
\beq
\label{k12a}
m\leq h(\la)\leq M \q {\rm and }\q 1/M\leq 1/h(\la)\leq 1/m.
\eeq
By Lemma \ref{kl4a}, for a given small enough $\epsilon>0$,
we can find two trigonometric polynomials
$t_1(\lambda)$ and $t_2(\lambda)$ to satisfy the following conditions:
\bea
\label{k13}
&&\|h-t_1\|_1 <\epsilon,\qq \|1/h-t_2\|_1<\epsilon,\\
&&
{m}/{2}< t_1(\lambda)< h(\lambda), \q  1/(2M)< t_2(\lambda)<1/{h(\lambda)},
\label{k14}
\eea
and hence
\bea
\label{k144}
{m}/{2}< t_1(\lambda)< h(\lambda)<{1}/{t_2(\lambda)}<2M.
\eea
Now in view of \eqref{OSm2K}, \eqref{k144} and Lemmas \ref{kl2} and \ref{kl3}, we obtain
\beq
\label{k15}
\nonumber
\liminf_{n\to\infty}\frac{\sigma_n^2(fh)}{\sigma_n^2(f)}
\geq\liminf_{n\to\infty}\frac{\sigma_n^2(ft_1)}{\sigma_n^2(f)}\geq G(t_1),
\eeq
and
\beq
\label{k16}
\nonumber
\limsup_{n\to\infty}\frac{\sigma_n^2(fh)}{\sigma_n^2(f)}
\leq\limsup_{n\to\infty}\frac{\sigma_n^2(f/t_2)}{\sigma_n^2(f)}\leq G(1/{t_2}).
\eeq
Therefore
\beq
\label{k155}
G(t_1)\leq \liminf_{n\to\infty}\frac{\sigma_n^2(fh)}{\sigma_n^2(f)}
\leq\limsup_{n\to\infty}\frac{\sigma_n^2(fh)}{\sigma_n^2(f)}\leq G(1/{t_2}).
\eeq
Next, in view of the first inequality in \eqref{k12a}, the last inequality in \eqref{k144}
and the second inequality in \eqref{k13}, we can write
\beq
\label{k156}
\|h-1/t_2\|_1=\|h/t_2(t_2-1/h)\|_1\leqslant2M^2\epsilon.
\eeq
From the first inequality in \eqref{k13} and \eqref{k156} we get
\beq
\label{k12ab}
\|t_1-1/t_2\|_1\leqslant\|t_1-h\|_1+\|h-1/t_2\|_1\leqslant \epsilon(1+2M^2).
\eeq
We now can write
$$0<\ln\frac{G(1/t_2)}{G(t_1)}=\ln G\left(\frac 1{t_1t_2}\right)
=\frac1{2\pi}\int_{-\pi}^{\pi}\ln \frac 1{t_1(\lambda)t_2(\lambda)}d\lambda
\leqslant\frac1{2\pi}\int_{-\pi}^{\pi}\left(\frac 1{t_1(\lambda)t_2(\lambda)}-1\right)d\lambda$$
\beq
\label{k157}
=\frac1{2\pi}\int_{-\pi}^{\pi}\frac 1{t_1(\lambda)}\left(\frac 1{t_2(\lambda)} -t_1(\lambda)\right)d\lambda
\leqslant\frac1{\pi m}\|t_1-1/t_2\|_1\leqslant\frac\epsilon {\pi m}(1+2M^2).
\eeq

Here the first relation follows from the inequality $1/t_2(\lambda)>t_1(\lambda)$
(see \eqref{k144}), the second from \eqref{gt1}, the third from \eqref{a2},
the fourth from the inequality $\ln x\leq x-1$ ($x>0$), the sixth from the first
inequality in \eqref{k14}, and the seventh from \eqref{k12ab}.

Thus, the quantities $G(t_1)$ and $G(1/t_2)$ can be made arbitrarily close.
On the other hand, in view of \eqref{OSm2K} and \eqref{k144}, we have
\beq
\label{k158}
G(t_1)\leqslant G(h)\leqslant G(1/t_2).
\eeq

Finally, from (\ref{k15}), (\ref{k157}) and (\ref{k158}) we obtain \eqref{k12}.
The inequality $G(h)>0$ follows from \eqref{k144}.
%Lemma \ref{kl4} is proved.
\end{proof}
Taking into account Proposition \ref {p4.1}(d), from Lemma \ref{kl4}
we obtain the following result.
\begin{cor}
\label{kc1}
If the sequence $\sigma_n(f)$ is weakly varying and $h(\lambda)\in B_+^-$,
then the sequence $\sigma_n(fh)$ is also weakly varying.
\end{cor}

\begin{lem}
\label{kl5}
Let the sequence $\sigma_n(f)$ be weakly varying, and let $h(\lambda)\in B^-$. Then
\beq
\label{k17}
\limsup_{n\to\infty}\frac{\sigma_n^2(fh)}{\sigma_n^2(f)}\leqslant G(h).
\eeq
\end{lem}
\begin{proof}
Observe that the function $h_\epsilon(\lambda)=h(\lambda)+\epsilon$
belongs to the class $B_+^-$, and $h_\epsilon(\lambda)\to h(\lambda)$ as $\epsilon\to0$.
Then we have the asymptotic relation
(see, Grenander and Szeg\"o \cite{GS}, Section 3.1 (d), p. 46):
\beq
\label{k177}
\lim_{\epsilon\to 0}G(h_\epsilon)=G(h).
\eeq
Hence, using Proposition \ref{pp3}(b) and Lemma \ref{kl4}, we obtain
$$\limsup_{n\to\infty}\frac{\sigma_n^2(fh)}{\sigma_n^2(f)}
\leq\lim_{n\to\infty}\frac{\sigma_n^2(fh_\epsilon)}{\sigma_n^2(f)}=G(h_\epsilon).$$
Passing to the limit as $\epsilon\to 0$, and taking into account \eqref{k177},
we obtain the desired inequality (\ref{k17}).
%Lemma \ref{kl5} is proved.
\end{proof}

As an immediate consequence of Lemma \ref{kl5}, we have the following result.
\begin{cor}
\label{kc2}
Let the sequence $\sigma_n(f)$ be weakly varying, and let $g(\lambda)\in B^-$
with $G(g)=0$. Then $\sigma_n(fg)=o(\sigma_n(f))$ as $n\to\infty$.
\end{cor}
Thus, multiplying singular spectral densities we obtain a spectral density
with higher "order of singularity".
\begin{lem}
\label{kl6}
Let the sequence $\sigma_n(f)$ be weakly varying, and let $h(\lambda)\in B_+$. Then
\beq
\label{k18}
\liminf_{n\to\infty}\frac{\sigma_n^2(fh)}{\sigma_n^2(f)}\geqslant G(h).
\eeq
\end{lem}
\begin{proof}
Let $h_l(\lambda)$ denote the truncation of $h(\lambda)$ at the level $l\in\mathbb{N}$:
$$h_l(\lambda)=\left\{
	\begin{array}{ll}
		h(\lambda), & h(\lambda)\leqslant l\\
		l, & h(\lambda)>l.
	\end{array}
	\right.
$$
Then by {\it monotone convergence theorem} of Beppo Levi
(see, e.g., Bogachev \cite{Bog}, Theorem 2.8.2, p. 130-131),
we have
\beq
\label{k19}
\lim_{l\to\infty}G(h_l)=G(h).
\eeq
Next, since $h_k(\la)\leq h(\la)$, in view of Proposition \ref{pp3}(b) and
Lemma \ref{kl4}, we get
$$\liminf_{n\to\infty}\frac{\sigma_n^2(fh)}{\sigma_n^2(f)}
\geqslant\lim_{n\to\infty}\frac{\sigma_n^2(fh_l)}{\sigma_n^2(f)}=G(h_l).$$
Hence passing to the limit as $l\to \f$, and taking into account (\ref{k19})
we obtain the desired inequality (\ref{k18}).
%Lemma \ref{kl6} is proved.
\end{proof}
As an immediate consequence of Lemma \ref{kl6}, we have the following result.
\begin{cor}
\label{kc3}
Let the sequence $\sigma_n(f)$ be weakly varying, $g(\lambda)\in B_+$ with $G(g)=\infty$,
and let $fg\in B$. Then $\sigma_n(f)=o(\sigma_n(fg))$ as $n\to\infty$.
\end{cor}

\sn{Proof of Theorem \ref{sT1}}
In this subsection we prove the main result of this section - Theorem \ref{sT1}.% and \ref{sT2}.
\begin{proof}[Proof of Theorem \ref{sT1}]
We have
\beq
\label{mk21}
\frac{\sigma_n^2(fg)}{\sigma_n^2(f)}=\frac{\sigma_n^2(fht_1/t_2)}{\sigma_n^2(f)}
=\frac{\sigma_n^2(fht_1/t_2)}{\sigma_n^2(fht_1)}\cd\frac{\sigma_n^2(fht_1)}{\sigma_n^2(fh)}
\cd\frac{\sigma_n^2(fh)}{\sigma_n^2(f)}.
\eeq
Next, by Lemma \ref{kl4} we have
\label{k21}
\beq
\lim_{n\to\infty}\frac{\sigma_n^2(fh)}{\sigma_n^2(f)}=G(h)>0.
\eeq
This, in view of Corollary \ref{kc1}, implies that the sequence $\sigma_n^2(fh)$
%as $\sigma_n^2(f)$
is also weakly varying. Therefore, by Lemma \ref{kl2}, we have
$$\liminf_{n\to\infty}\frac{\sigma_n^2(fht_1)}{\sigma_n^2(fh)}\geq G(t_1).$$
On the other hand, since $t_1(\lambda)\in B^-$, then according to  Lemma \ref{kl5},
we get
$$\limsup_{n\to\infty}\frac{\sigma_n^2(fht_1)}{\sigma_n^2(fh)}\leq G(t_1).$$
Therefore
\beq
\label{k22}
\lim_{n\to\infty}\frac{\sigma_n^2(fht_1)}{\sigma_n^2(fh)}=G(t_1)>0
\eeq
This implies that the sequence $\sigma_n^2(fht_1)$ is also weakly varying.
Hence we can apply Lemma \ref{kl3}, to obtain
$$\limsup_{n\to\infty}\frac{\sigma_n^2(fht_1/t_2)}{\sigma_n^2(fht_1)}
\leqslant G(1/t_2).$$
Next, it is easy to see that $1/t_2\in B_+$. Hence, according to Lemma \ref{kl6},
we obtain
$$\liminf_{n\to\infty}\frac{\sigma_n^2(fht_1/t_2)}{\sigma_n^2(fht_1)}
\geqslant G(1/t_2).$$
Therefore
\beq
\label{k23}
\lim_{n\to\infty}\frac{\sigma_n^2(fht_1/t_2)}{\sigma_n^2(fht_1)}=G(1/t_2).
\eeq
Finally, combining the relations  (\ref{mk21}) - (\ref{k23}), we obtain
$$\lim_{n\to\infty}\frac{\sigma_n^2(fg)}{\sigma_n^2(f)}
=G(1/t_2)G(t_1)G(h)=G(ht_1/t_2)=G(g)>0.$$
Theorem \ref{sT1} is proved.
\end{proof}

\sn{Examples}

In this section we discuss examples demonstrating the result obtained in Theorem \ref{sT1}.
In these examples we assume that $\{X(t),$ $t\in\mathbb{Z}\}$ is a stationary deterministic
process with a spectral density $f(\la)$ satisfying the conditions of Theorem \ref{sT1},
and the function $g$ is given by formula \eqref{g}.
To compute the geometric means we use the properties stated in Proposition \ref{p4.2}(a).
\begin{exa}
\label{ex41}
{\rm Let the function $g(\la)$ be as in \eqref{g} with $h(\la)=c>0$ and
$t_1(\la)=t_2(\la)=1$, that is, $g(\la)=c>0$. Then for the geometric mean
$G(g)$ we have
\beq
\label{g0}
G(g)=G(c)=c,
\eeq
and in view of \eqref{k7}, we get
\beq
\label{k71}
\nonumber
\lim_{n\to\f}\frac{\si^2_{n}(fg)}{\si^2_{n}(f)} =G(g)=c.
\eeq
Thus, multiplying the spectral density $f(\la)$ by a constant $c>0$
changes asymptotically the prediction error by $c$ times.}
\end{exa}
\begin{exa}
\label{ex42}
{\rm Let the function $g(\la)$ be as in \eqref{g} with $h(\la)=e^{\I(\la)}$,
where $\I(\la)$ is an arbitrary odd function, and let
$t_1(\la)=t_2(\la)=1$, that is, $g(\la)=e^{\I(\la)}$.
Then for the geometric mean $G(g)$ we have
\beq
\label{g2}
G(g)=G(e^{\I(\la)})=\exp\left\{\frac1{2\pi}\inl\ln g(\la)\,d\la \right\}
=\exp\left\{\frac1{2\pi}\inl\I(\la)\,d\la \right\}=e^0=1,
\eeq
and in view of \eqref{k7}, we get
\beq
\label{k72}
\nonumber
\lim_{n\to\f}\frac{\si^2_{n}(fg)}{\si^2_{n}(f)} =G(g)=1.
\eeq
Thus, multiplying the spectral density $f(\la)$ by the function $e^{\I(\la)}$ with odd
$\I(\la)$ does not change the asymptotic behavior of the prediction error.}
\end{exa}

\begin{exa}
\label{ex43}
{\rm Let the function $g(\la)$ be as in \eqref{g} with $h(\la)=\la^2+1$ and
$t_1(\la)=t_2(\la)=1$, that is, $g(\la)=\la^2+1$. Then for the geometric mean $G(g)$
by direct calculation we obtain
\beq
\label{g3}
G(g)=\exp\left\{\frac1{2\pi}\inl\ln (\la^2+1)\,d\la \right\}
=\exp\{\ln(1+\pi^2)-2+\frac2\pi\arctan\pi\}\approx 3.3,
\eeq
and in view of \eqref{k7}, we get
\beq
\label{k72a}
\nonumber
\lim_{n\to\f}\frac{\si^2_{n}(fg)}{\si^2_{n}(f)} =G(g)
=\exp\{\ln(1+\pi^2)-2+\frac2\pi\arctan\pi\}\approx 3.3.
\eeq
Thus, multiplying the spectral density $f(\la)$ by the function $\la^2+1$
increases asymptotically the prediction error approximately by 3.3 times.}
\end{exa}

\begin{exa}
\label{ex44}
{\rm Let the function $g(\la)$ be as in \eqref{g} with $h(\la)=t_2(\la)=1$,
and $t_1(\la)=\sin^{2k}(\la-\la_0)$, where $k\in\mathbb{N}$ and $\la_0$
is an arbitrary point from $[-\pi,\pi]$, that is, $g(\la)=\sin^{2k}(\la-\la_0)$.
To compute the geometric mean $G(g)$, we first find the algebraic
polynomial $s_2(z)$ in the Fej\'er-Riesz representation \eqref{ss2}
of the nonnegative trigonometric polynomial $\sin^2(\la-\la_0)$ of degree 2.
For any $\la_0\in[-\pi,\pi]$ we have
\bea
\label{k731}
\nonumber
\sin^2(\la-\la_0) &=&|\sin(\la-\la_0)|^2= \left|\frac{e^{i(\la-\la_0)}-e^{-i(\la-\la_0)}}{2i}\right|^2\\
\nonumber
&=&\left|\frac12(e^{2i(\la-\la_0)}-1)\right|^2
=\left|\frac12(e^{-2i\la_0}e^{2i\la}-1)\right|^2=\left|s_2(e^{i\la})\right|^2,
\eea
where
\beq
\label{k741}
s_2(z)=\frac12(e^{-2i\la_0}z^2-1).
\eeq
Therefore, by \eqref{k9} and \eqref{k741}, we have
\beq
\label{k742g}
G(\sin^{2}(\la-\la_0))=|s_2(0)|^2=\left(\frac12\right)^2=\frac14.
\eeq
Now, in view of Proposition \ref{p4.2}(a) and \eqref{k742g}, for the geometric mean of $g(\la)=t_1(\la)=\sin^{2k}(\la-\la_0)$ ($k\in\mathbb{N}$), we obtain
\beq
\label{g4}
G(g)=G(\sin^{2k}(\la-\la_0))=G^k(\sin^2(\la-\la_0))=\frac1{4^k},
\eeq
and in view of \eqref{k7}, we get
\beq
\label{k76}
\nonumber
\lim_{n\to\f}\frac{\si^2_{n}(fg)}{\si^2_{n}(f)} =G(g)=\frac1{4^k}.
\eeq
Thus, multiplying the spectral density $f(\la)$ by the nonnegative trigonometric polynomial
$\sin^{2k}(\la-\la_0)$  of degree $2k$ ($k\in\mathbb{N}$), yields a
$4^k$ times asymptotic reduction of the prediction error.}
\end{exa}
\begin{exa}
\label{ex45}
{\rm Let the function $g(\la)$ be as in \eqref{g} with $h(\la)=t_1(\la)=1$,
and $t_2(\la)=\sin^{2l}(\la-\la_0)$, where $l\in\mathbb{N}$ and $\la_0$
is an arbitrary point from $[-\pi,\pi]$, that is, $g(\la)=\sin^{-2l}(\la-\la_0)$.
Then, in view of the third equality in \eqref{gt1} and \eqref{g4}
for the geometric mean $G(g)$ we have
\beq
\label{g5}
G(g)=G(\sin^{-2l}(\la-\la_0))=G^{-1}(\sin^{2l}(\la-\la_0)) ={4^l},
\eeq
and in view of \eqref{k7}, we get
\beq
\label{k76g}
\nonumber
\lim_{n\to\f}\frac{\si^2_{n}(fg)}{\si^2_{n}(f)} =G(g)={4^l}.
\eeq
Thus, dividing the spectral density $f(\la)$ by the nonnegative trigonometric polynomial
$\sin^{2l}(\la-\la_0)$  of degree $2l$ ($l\in\mathbb{N}$), yields
a $4^l$ times asymptotic increase of the prediction error.}
\end{exa}
Notice that the values of the geometric mean $G(g)$ obtained in
\eqref{g4} and \eqref{g5} do not depend on the choice of the point
$\la_0\in[-\pi,\pi]$.

Putting together Examples \ref{ex41} - \ref{ex45} and using Proposition \ref{p4.2}(a)
we have the following summary example.
\begin{exa}
\label{ex46}
{\rm Let $\{X(t), \,t\in\mathbb{Z}\}$ be a stationary deterministic
process with a spectral density $f(\la)$ satisfying the conditions
of Theorem \ref{sT1}. Let $h(\la)=ce^{\I(\la)}(\la^2+1)$, $t_1(\la)=\sin^{2k}(\la-\la_1)$
and $t_2(\la)=\sin^{2l}(\la-\la_2)$, where $c$ is an arbitrary positive constant,
$\I(\la)$ is an arbitrary odd function and $\la_1, \la_2$ are arbitrary points from $[-\pi,\pi]$.
Let the function $g(\la)$ be defined as in \eqref{g}, that is,
\beq
\label{g6}
g(\la)=h(\la)\cd\frac{t_1(\la)}{t_2(\la)}=ce^{\I(\la)}(\la^2+1)
\frac{\sin^{2k}(\la-\la_1)}{\sin^{2l}(\la-\la_2)}.
\eeq
Then, in view of Proposition \ref{p4.2}(a) and relations \eqref{g0}--\eqref{g3}
and \eqref{g4}--\eqref{g6}, we have
\bea
\label{g7}
\nonumber
G(g)&=&
G(h)\frac{G(t_1)}{G(t_2)}
=G(c)G(e^{\I})G(\la^2+1)G(\sin^{2k}(\la-\la_1))G(\sin^{-2l}(\la-\la_2))\\
&=&(c)(1)\exp\{\ln(1+\pi^2)-2+\frac2\pi\arctan\pi\}(4^{-k})(4^{l})
\approx 3.3c4^{l-k},
\eea
and in view of \eqref{k7} and \eqref{g7}, we get
\beq
\nonumber
\lim_{n\to\f}\frac{\si^2_{n}(fg)}{\si^2_{n}(f)} =G(g)
\approx 3.3c4^{l-k}.
\eeq
}
\end{exa}

\n
{\bf Acknowledgment.}
%We would like to thank the anonymous referees for their useful suggestions.
Murad S. Taqqu was supported in part by a Simons Foundation grant 569118 at Boston University.

%\cite{An}

\bigskip
\small
\noindent Nikolay M. Babayan:\\ Russian-Armenian University, Yerevan, Armenia,
e-mail: nmbabayan@gmail.com.\\
Mamikon S. Ginovyan:\\ Boston University, Boston, USA, e-mail: ginovyan@math.bu.edu.\\
Murad S. Taqqu:\\ Boston University, Boston, USA, e-mail: murad@bu.edu.


\begin{thebibliography}{999}

\baselineskip3.9mm
\small

\bibitem{Ah}
Ahlfors, L. V. {\it Complex Analysis: An Introduction to the Theory of Analytic
Functions of One Complex Variable}. Mcgraw-Hill, New York, 1969.

\bibitem{Bax}
Baxter G.   An Asymptotic Result for the Finite Predictor,
{\it Math. Scand.,\/}  {\bf ~10},  137 -- 144, 1962.

\bibitem{Bb-1}
Babayan, N. M. On the asymptotic behavior of prediction error,
{\it J. of Soviet Mathematics}, {\bf ~27}(6), 3170 -- 3181, 1984.

\bibitem{Bb-2}
Babayan, N. M.  On asymptotic behavior of the prediction error in the
singular case, {\it Theory Probab. Appl}, {\bf 29}(1),
147 -- 150, 1985.

\bibitem{Bin1}
Bingham, N.H. Szeg\"o's theorem and its probabilistic
descendants. {\it Probability Surveys}, {\bf 9}, 287 -- 324, 2012.

\bibitem{Bog}
Bogachev, V.  {\it Measure Theory}, vol. I, Springer, Berlin, 2007.

\bibitem{Dav1}
Davisson, L. D.  Prediction of time series from finite past.
{\it J. Soc. Indust. Appl. Math.}, {\bf 13}(3), 819-826, 1965.

\bibitem{Dev1}
Devinatz A.   Asymptotic estimates for the finite predictor,
{\it Math. Scand.,\/} {\bf ~15}, 111--120, 1964.

\bibitem{F}
Fekete, M.
\"Uber den transfiniten Durchmesser ebener Punktmengen.
{\it Zweite Mitteilung. Math. Z.}, {\bf 32}, 215--221, 1930.

\bibitem{For}
Fortus, M. I.  Prediction of a stationary time series with
the spectrum vanishing on an interval.
{\it Akademiia Nauk SSSR, Izvestiia, Fizika Atmosfery i Okeana},
{\bf 26}, 1267--1274, 1990.

\bibitem{G-4}
Ginovian, M. S. Asymptotic behavior of the prediction error for
stationary random sequences,
{\it Journal of Contemporary Math. Anal.},  {\bf 4}(1), 14 -- 33, 1999.

\bibitem{Gol-1}
Golinskii B. L., On asymptotic behavior of the prediction error,
{\it Theory Probab. and appl.},  {\bf 19}(4), 724 -- 739, 1974.

\bibitem{GI}
Golinskii B. L. and Ibragimov I. A.,  On G. Szeg\"o limit theorem,
{\it Izv. AN SSSR, ser. Matematika}, {\bf 35}, 408 -- 427, 1971.

\bibitem{GMG}
Goluzin, G. M. {\it Geometric Theory of Functions of a Complex Variable}.
Providence, Amer. Math. Soc., 1969.

\bibitem{GR1}
Grenander, U., Rosenblatt, M.  An Extension of a Theorem of G. Szeg\"o
and its Application to the Study of Stochastic Processes,
{\it Trans. Amer. Math. Soc.}, {\bf 76}, 112 -- 126, 1954.

\bibitem{GS}
Grenander, U., Szeg\"o,  G.  {\it Toeplitz Forms and Their Applications}.
University of California Press, Berkeley and Los Angeles, 1958.

\bibitem{HS}
Helson, H. and Szeg\"o, G., A problem in prediction theory.
{\it Acta Mat. Pura Appl.}, {\bf 51}, 107 -- 138, 1960.

\bibitem{Hir}
Hirschman, I. I. Finite sections of Wiener-Hopf equations and Szeg\"o
polynomials. {\it Journal of Mathematical Analysis and Applications},
{\bf 11},  290--320, 1965.

\bibitem{I-2}
Ibragimov, I. A.  On asymptotic behavior of the prediction error,
{\it Probab. Theory and appl.},  {\bf 9}(4), 695 -- 703, 1964.

\bibitem{IS}
Ibragimov, I. A., Solev, V. N. The asymptotic behavior of the prediction
error of a stationary sequence with the  spectral density function of
a special  form. {\it Probab. Theory and Appl.}, {\bf 13}(4), 746 -- 750, 1968.

\bibitem{In2}
Inoue A.
Asymptotic behavior for partial autocorrelation functions
of fractional ARIMA processes, {\it The Annals of Applied Probability},
{\bf 12}(4), 1471 - 1491, 2002.

\bibitem{Kol1}
Kolmogorov A. N. Stationary sequences in a Hilbert space,
{\it Bull. Moscow State University}, {\bf 2}(6), 1--40, 1941.

\bibitem{Kol2}
Kolmogorov A. N. Inerpolation and Extrapolation of stationary
random sequences, {\it Izv. Akad. Nauk SSSR. Ser. Tat.}, {\bf 5}, 3--14, 1941.

\bibitem{Maz}
Mazurkievicz S.: Un theoreme sur les polynomes.
{\it Ann. Soc. Polon.Math.}, {\bf 18},  113 -- 118, 1945.

\bibitem{Po}
Pourahmadi, M. {\it Fundamentals of Time Series Analysis and
Prediction Theory}. New York, Wiley, 2001.

\bibitem{Rob}
Robinson, R. M.
On the transfinite diameters of some related sets.
{\it Math. Z.} {\bf 108}, 377--380 1969.

\bibitem{Ros}
Rosenblatt, M.  Some Purely Deterministic Processes,
{\it  J. of Math. and Mech.,\/} {\bf ~6} (6), 801 -- 810, 1957.
(Reprinted in: {\it Selected works of Murray Rosenblatt},
Davis R.A, Lii K.-S., Politis D.N. eds. Springel, New York, 124-133, 2011).
\bibitem{R}
Rozanov, Yu. A., {\it Stationary random processes}. Holden-Day, San Francisco, 1967.

\bibitem{Saf}
Saff, E.B.
Logarithmic Potential Theory with Applications to Approximation Theory.
{\it  Surveys in Approximation Theory}, {\bf 5}, 165 -- 200, 2010.

\bibitem{Sz}
Szeg\"o, G.
Ein Grenzwertsatz \"uber die Toeplitzschen Determinanten einer
reellen positiven Funktion, {\it Math. Ann.} {\bf 76}, 490--503, 1915.

\bibitem{Sz1}
Szeg\"o, G. {\it Orthogonal Polynomials}, Amer. Math. Soc. Colloq. Publ., 23,
American Mathematical Society, Providence, RI, 1939; 3rd edition, 1967.

\bibitem{Tsu}
Tsuji, M. {\it Potential theory in modern function theory}, Chelsea Pub. Co;
2nd edition, New York, 1975.

\bibitem{Wr1}
Wiener, N. {\it Extrapolation, interpolation and smoothing of stationary
time series. With engineering applications}. MIT Press/Wiley, 1949.

\end{thebibliography}
\end{document}